\documentclass[10pt]{article}

\usepackage{amsmath, amssymb, amstext, amsthm}
\usepackage{graphicx}
\usepackage[utf8]{inputenc}
\usepackage[english]{babel}
\usepackage{fancyhdr}
\usepackage{relsize}
\usepackage{wrapfig} 
\usepackage{lipsum} 
\usepackage{titlesec}
\usepackage{esvect} 
\usepackage[margin=1.00in]{geometry}
\usepackage[symbol]{footmisc}
\usepackage{perpage}
\usepackage{float}
\usepackage{subfig}
\usepackage{commath}
\usepackage{mathtools}
\usepackage{authblk} 
\usepackage{setspace}
\usepackage{caption} 
\usepackage{multicol} 
\usepackage{multirow} 
\usepackage{lastpage}
\usepackage{color}

\numberwithin{equation}{section}

\theoremstyle{definition}
\newtheorem{prop}{Proposition}[]
\newtheorem{remark}{Remark}[]
\newtheorem{example}{Example}[section]

\newcommand{\xdashthick}[1][46em]{\rule[0.5ex]{#1}{1.5pt}}
\newcommand{\xdashthin}[1][46em]{\rule[0.5ex]{#1}{0.6pt}}

\title{An Eulerian-Lagrangian Runge-Kutta finite volume (EL-RK-FV) method for solving convection and convection-diffusion equations}
 
\author[1]{Joseph Nakao}
\author[1]{Jiajie Chen}
\author[1]{Jingmei Qiu}
\affil[1]{Department of Mathematical Sciences, University of Delaware, Newark, DE 19711}
\date{}

\begin{document}

\maketitle


\begin{abstract}
    \noindent We propose a new Eulerian-Lagrangian Runge-Kutta finite volume method for numerically solving convection and convection-diffusion equations. Eulerian-Lagrangian and semi-Lagrangian methods have grown in popularity mostly due to their ability to allow large time steps. Our proposed scheme is formulated by integrating the PDE on a space-time region partitioned by approximations of the characteristics determined from the Rankine-Hugoniot jump condition; and then rewriting the time-integral form into a time differential form to allow application of Runge-Kutta (RK) methods via the method-of-lines approach. The scheme can be viewed as a generalization of the standard Runge-Kutta finite volume (RK-FV) scheme for which the space-time region is partitioned by approximate characteristics with zero velocity. The high-order spatial reconstruction is achieved using the recently developed weighted essentially non-oscillatory schemes with adaptive order (WENO-AO); and the high-order temporal accuracy is achieved by explicit RK methods for convection equations and implicit-explicit (IMEX) RK methods for convection-diffusion equations. Our algorithm extends to higher dimensions via dimensional splitting. Numerical experiments demonstrate our algorithm's robustness, high-order accuracy, and ability to handle extra large time steps.
\end{abstract}

\noindent\textbf{Keywords:} Eulerian-Lagrangian, semi-Lagrangian, convection-diffusion equation, WENO-AO, IMEX Runge-Kutta


\section{Introduction}
In this paper, we are concerned with numerically solving convection-diffusion equations of the form
\begin{equation}\label{CDeqn}
\begin{cases}
    u_t + \nabla\cdot\mathbf{F}(u) = \epsilon\Delta u + g(\mathbf{x},t),&\mathbf{x}\in\mathcal{D},\quad t>0,\\
    u(\mathbf{x},t=0)=u_0(\mathbf{x}),&\mathbf{x}\in\mathcal{D},
\end{cases}
\end{equation}
where $\epsilon\geq 0$. We propose an Eulerian-Lagrangian Runge-Kutta finite volume (EL-RK-FV) scheme utilizing weighted essentially non-oscillatory (WENO) schemes with adaptive order (WENO-AO) for spatial reconstruction. The proposed method is designed for one-dimensional problems of the form \eqref{CDeqn} and extended to higher-dimensional problems via dimensional splitting.\\
\indent Eulerian-Lagrangian (EL) and semi-Lagrangian (SL) schemes \cite{Celia1990,Russell2002,Xiu2001} have proven to be computationally effective when solving hyperbolic problems because of their ability to employ high spatial resolution schemes while admitting very large CFL with numerical stability. Generally speaking, an EL or SL method involves working on a background grid so that high-order spatial resolutions can be used (the Eulerian part), and tracing characteristics of the cell boundaries backward/forward in time to relax the CFL constraint (the Lagrangian part). Such methods have been developed in a wide variety of frameworks: discontinuous Galerkin \cite{Cai2017,Ding2020,Qiu2011b,Rossmanith2011}, finite difference \cite{Carrillo2007,Chen2021,Huot2003,Li2022,Qiu2010,Qiu2011,Xiong2019}, finite volume \cite{Abreu2017,Benkhaldoun2015,Crouseilles2010,Filbet2001,Huang2012,Huang2017}. Although the SL and EL frameworks are similar in spirit, the SL framework assumes exact characteristic tracing and hence poses difficulties when considering nonlinear problems. Two other methods similar to EL and SL methods are the arbitrary Lagrangian-Eulerian (ALE) methods where an arbitrary mesh velocity not necessarily aligned with the fluid velocity is defined \cite{Boscheri2014,Boscheri2013,Boscheri2015,Boscheri2017,Donea2004,Hirt1974,Peery2000}, and moving mesh methods where the PDE is first evolved in time and then followed by some mesh-redistribution procedure \cite{Li2001,Li1997,Luo2019,Russell2002,Stockie2001,Tang2003}. The main difference between these two methods and the previously mentioned methods is that they move the mesh adaptively to focus resolving the solution around sharp transitions; whereas EL and SL methods evolve the equation by following characteristics.\\
\indent Our goal in this paper is to develop a new high-order EL method in the finite volume framework using method-of-lines (MOL) RK time discretizations. Finite volume methods are attractive since they are naturally mass conservative, easy to physically interpret, and modifiable for nonuniform grids. Similar to the recent developments made by Huang, Arbogast, and Qiu \cite{Huang2012,Huang2017}, we use approximate characteristics to define a traceback space-time region, and then use WENO reconstructions to evaluate the modified flux. Huang and Arbogast developed a re-averaging technique that allows high-order reconstruction of the solution at arbitrary points by applying a standard WENO scheme \cite{Cockburn1997,Shu2009} over a uniform reconstruction grid that is defined separately. They then used a natural continuous extension \cite{Zennaro1986} of Runge-Kutta schemes, which requires the solution at several Gaussian nodes of the interval $[t^n,t^{n+1}]$, to evolve the solution along the approximate characteristics. \\
\indent The novelty of our proposed method is twofold: (1) the partition of space-time regions formed by linear approximations of the characteristic curves, and (2) integrating the differential equation over the partitioned space-time regions, followed by rewriting the space-time integral form of the equation into a spatial-integral time-differential form. In this way, a MOL RK type method can be directly applied for time discretization, thus avoiding the need to use a natural continuous extension of RK schemes. To be more precise, we construct linear approximate characteristics by using the Rankine-Hugoniot jump condition to define the traceback space-time regions. If the linear approximate characteristics are defined with zero velocity, then the proposed EL-RK-FV scheme reduces to the standard RK-FV method. Whereas, when linear space-time curves adequately approximate the exact characteristics, a large time stepping size is still permitted. We use WENO-AO to perform a solution remapping of the uniform cell averages onto the possibly nonuniform traceback cells. The recently developed WENO-AO schemes \cite{Arbogast2020,Balsara2016,Balsara2020} are robust and guarantee the existence of the linear weights at arbitrary points. We note that Chen, et al. used WENO-AO schemes in the SL framework \cite{Chen2021}, and Huang and Arbogast have recently used WENO-AO schemes in the Eulerian framework \cite{Arbogast2019,Arbogast2020}. RK methods are used to evolve the MOL system along the approximate characteristics. Explicit RK methods, such as the strong stability-preserving (SSP) RK methods \cite{Gottlieb2001}, are used for convection equations; and implicit-explicit (IMEX) RK methods \cite{Ascher1997,Conde2017,Higueras2014} are used for convection-diffusion equations. In the latter case, the non-stiff convective term is treated explicitly and the stiff diffusive term is treated implicitly. Dimensional splitting is used to extend the one-dimensional algorithm to solve multi-dimensional problems. The proposed method is high-order accurate, capable of resolving discontinuities without oscillations, mass conservative, and stable with large time stepping sizes.\\
\indent The paper is organized as follows. We discuss the EL-RK-FV algorithm for pure convection problems in Section 2. In Section 3, we discuss the EL-RK-FV algorithm, coupled with IMEX RK schemes, for convection-diffusion equations. Numerical performance of the EL-RK-FV algorithm is shown in Section 4 by applying the algorithm to several linear and nonlinear test problems. Conclusions are made in Section 5. Appendices are listed after the references.


\section{The EL-RK-FV method for pure convection problems}

The spirit of the EL-RK-FV method is best demonstrated by starting with a pure convection problem in one dimension, i.e., equation \eqref{CDeqn} with $\epsilon=0$ and $g(x)=0$. Let the flux be denoted $f(x)$. We first discuss the formulation of the scheme in Section~\ref{sec2.1}, followed by discussion of high-order spatial reconstruction in Section~\ref{sec2.2} and time discretization in Section~\ref{sec2.3}.

\subsection{
Scheme formulation
}
\label{sec2.1}

We discretize the spatial domain $[a,b]$ into $N_x$ intervals with $N_x+1$ uniformly distributed nodes
\[a=x_{\frac{1}{2}}<x_{\frac{3}{2}}<...<x_{N_x-\frac{1}{2}}<x_{N_x+\frac{1}{2}}=b.\]
Define the cells $I_j\coloneqq[x_{j-\frac{1}{2}},x_{j+\frac{1}{2}}]$ with centers $x_j=(x_{j-\frac{1}{2}}+x_{j+\frac{1}{2}})/2$ and widths $\Delta x_j=\Delta x$ for $j=1,...,N_x$. We let
\begin{equation}
    \label{CFL}
    \Delta t = \frac{CFL\Delta x}{\text{max}|f'(u)|},
\end{equation}
where $CFL$ defines the time stepping size. In contrast to Eulerian methods, which evolve the solution on a stationary mesh, our EL algorithm proposes tracing the characteristics \textit{backwards} in time from $t^{n+1}$ to $t^n$ to partition a set of space-time regions based on the computational grid. Since tracing characteristics in the nonlinear case is often nontrivial, we consider computing \textit{approximations of characteristics} that are linear space-time curves. In particular, the \textit{approximate characteristic speeds} at nodes $x_{j+\frac{1}{2}}$ and time $t^{n+1}$ are defined using the Rankine-Hugoniot jump condition,
\begin{equation}\label{RHjump}
	\nu_{j+\frac{1}{2}} = 
	\begin{cases}
		\mathlarger{\frac{f(\overline{u}_{j+1}) - f(\overline{u}_j)}{\overline{u}_{j+1} - \overline{u}_j}},&\overline{u}_{j+1}\neq \overline{u}_j,\\
		f'(\overline{u}_j),&\overline{u}_{j+1}=\overline{u}_j,
	\end{cases}
\end{equation}
where $\overline{u}_j$ and $\overline{u}_{j+1}$ are the cell averages at time $t^n$. In practice we tested if $|\overline{u}_{j+1}-\overline{u}_j|<\epsilon$, with $\epsilon=1.0e-8$, in which case we took $\overline{u}$ to be the average of $\overline{u}^-$ and $\overline{u}^+$. As seen in Figure \ref{spacetime}, the (upstream) traceback nodes can be defined by
\begin{equation}
    x_{j+\frac{1}{2}}^*\coloneqq x_{j+\frac{1}{2}} - \nu_{j+\frac{1}{2}}\Delta t,\quad j=1,...,N.
\end{equation}
Define $\tilde{x}_{j+\frac{1}{2}}(t)=x^*_{j+\frac{1}{2}}+\nu_{j+\frac{1}{2}}(t-t^n)$ with $t^n\leq t\leq t^{n+1}$, for $j=0, 1,2,...,N$. The approximate characteristics are given by the space-time tracelines
\[\mathcal{S}_{left}\coloneqq\{(\tilde{x}_{j-\frac{1}{2}}(t),t)\text{ : }t^n\leq t\leq t^{n+1}\}\quad\text{and}\quad\mathcal{S}_{right}\coloneqq\{(\tilde{x}_{j+\frac{1}{2}}(t),t)\text{ : }t^n\leq t\leq t^{n+1}\}.\]
Note that $x_{j+\frac{1}{2}}^*=\tilde{x}_{j+\frac{1}{2}}(t^n)$, $x_{j+\frac{1}{2}}=\tilde{x}_{j+\frac{1}{2}}(t^{n+1})$, and the (upstream) traceback cells $I_j^*=\tilde{I}_j(t^n)$ are in general nonuniform. We define the space-time domain $\Omega_j$ as the region bounded by $I_j$, $I_j^*$, $\mathcal{S}_{left}$, and $\mathcal{S}_{right}$.
\begin{figure}[h!]
	\centering
	\includegraphics[width=0.6\textwidth]{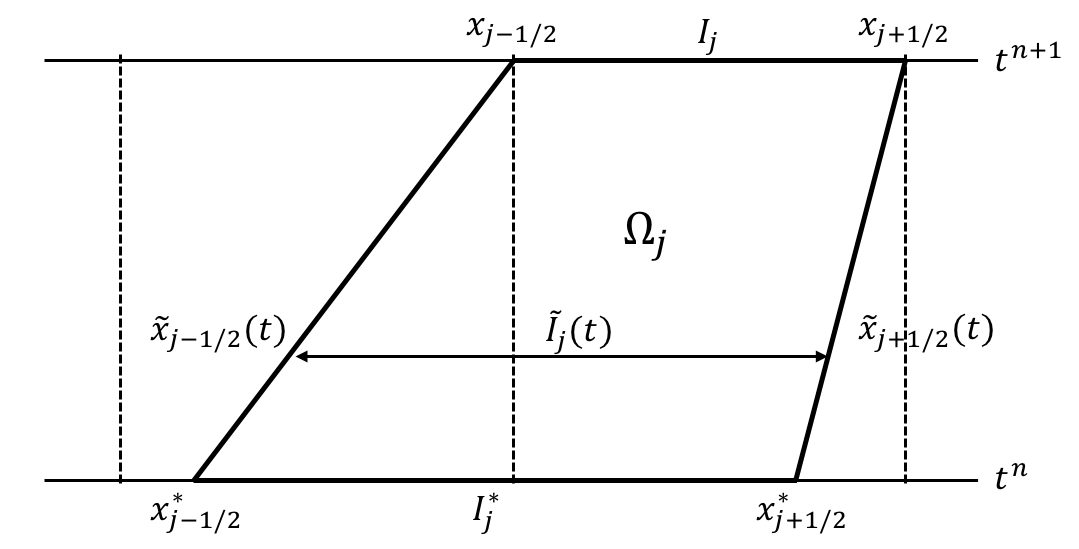}
	\caption{The space-time region $\Omega_j$.}
	\label{spacetime}
\end{figure}
\\
With the constructed space-time region $\Omega_j$, we rewrite
the one-dimensional pure convection problem in divergence form $\nabla_{t,x}\cdot(u,f(u))^T=0$, integrate it over $\Omega_j$, and apply the divergence theorem to get
\begin{align}
\begin{split}\label{semidiscrete0}
	&\int_{I_j}{u(x,t^{n+1})dx} - \int_{I_j^*}{u(x,t^n)dx}\\
	=&- \left[\int_{t^n}^{t^{n+1}}{\big(f(u(\tilde{x}_{j+\frac{1}{2}}(t),t))-\nu_{j+\frac{1}{2}}u(\tilde{x}_{j+\frac{1}{2}}(t),t)\big)dt} - \int_{t^n}^{t^{n+1}}{\big(f(u(\tilde{x}_{j-\frac{1}{2}}(t),t))-\nu_{j-\frac{1}{2}}u(\tilde{x}_{j-\frac{1}{2}}(t),t)\big)dt}\right].
\end{split}
\end{align}
We rewrite the time-integral form \eqref{semidiscrete0} into the time-differential form
to get
\begin{equation}\label{semidiscrete1}
	\frac{d}{dt}\int_{\tilde{I}_j(t)}{u(x,t)dx} = -\Big[F_{j+\frac{1}{2}}(t) - F_{j-\frac{1}{2}}(t)\Big],
\end{equation}

\noindent where $F_{j+\frac{1}{2}}(t)\coloneqq f(u(\tilde{x}_{j+\frac{1}{2}}(t),t))-\nu_{j+\frac{1}{2}}u(\tilde{x}_{j+\frac{1}{2}}(t),t)$ is called the \textit{modified flux function}. Choosing any appropriate monotone numerical flux function
\begin{equation}
    \label{numflux}
    \hat{F}_{j+\frac{1}{2}}(t) = \hat{F}_{j+\frac{1}{2}}(u_{j+\frac{1}{2}}^{-},u_{j+\frac{1}{2}}^{+};t)
\end{equation}
(e.g., Lax-Friedrichs flux) for the modified flux function, equation \eqref{semidiscrete1} can be rewritten as the following semi-discrete finite volume scheme:
\begin{equation}\label{semidiscrete2}
	\frac{d}{dt}\int_{\tilde{I}_j(t)}{u(x,t)dx} = -\Big[\hat{F}_{j+\frac{1}{2}}(t) - \hat{F}_{j-\frac{1}{2}}(t)\Big].
\end{equation}
Here, the starting condition is obtained by a solution remapping onto a trackback grid, discussed in Section~\ref{sec2.2}; $\hat{F}_{j+\frac{1}{2}}(t)$ are computed from \eqref{numflux} with reconstructed values from neighboring cell averages at time $t$, discussed in Section~\ref{sec2.3}. We evolve equation \eqref{semidiscrete2} by the MOL approach with explicit RK schemes, discussed in Section~\ref{sec2.4}.

\subsection{Solution remapping onto a traceback grid}
\label{sec2.2}

Referring back to Figure \ref{spacetime} and equation \eqref{semidiscrete0}, we see that the solution will need to be projected onto the traceback grid in order to compute the \textit{traceback cell averages} 
\begin{equation}
    \tilde{u}_j^n \coloneqq \frac{1}{\Delta x_j^*}\int_{I_j^*}{u(x,t^n)dx},\qquad j=1,2,...,N,
\end{equation}
that is, the starting condition for \eqref{semidiscrete2}. Hence, we desire a procedure for an \textit{integral reconstruction} that --in general-- uses known uniform cell averages $\{ \overline{u}_j(t)\text{ : }j=1,2,...,N\}$ to approximate the desired cell averages $\{ \tilde{u}_j(t)\text{ : }j=1,2,...,N \}$ defined by
\begin{equation}
    \tilde{u}_j(t) \coloneqq \frac{1}{\Delta\tilde{x}_j(t)}\int_{\tilde{I}_j(t)}{u(x,t)dx},\qquad j=1,2,...,N.
\end{equation}
Unless otherwise stated, overlines (e.g., $\overline{u}_j(t)$) denote \textit{uniform} cell averages, and tildes (e.g., $\tilde{u}_j(t)$) denote \textit{nonuniform} cell averages.\\
\ \\
Since discontinuities and sharp gradients can occur when solving pure convection problems, we use high-resolution schemes to control spurious oscillations, e.g., weighted essential non-oscillatory (WENO) methods \cite{Balsara2016,Cockburn1997,Levy1999,Shu2009}. Again, we assume to only be given the uniform cell averages at some time $t$. One might apply the well known WENO procedure presented in \cite{Cockburn1997,Shu2009} to obtain the \textit{reconstruction polynomials} $\mathcal{R}_j(x\in I_j)$ in terms of the neighboring uniform cell averages (at time $t$) $\overline{u}_i$, $i=j-p,...,j+q$. Referring to Figure \ref{spacetime} for the sake of demonstration, we might then approximate the cell averages $\tilde{u}_j^*$ by
\begin{equation}
    \label{WENO}
    \tilde{u}_j^* \approx \frac{1}{\Delta x_j^*}\left(\int_{I_j^*\cap I_{j-1}}{\mathcal{R}_{j-1}(x)dx} + \int_{I_j^*\cap I_{j}}{\mathcal{R}_{j}(x)dx}\right).
\end{equation}
However, the linear weights in the WENO reconstruction are not guaranteed to exist or be positive at arbitrary points. To alleviate this issue, we instead use the WENO schemes with adaptive order (WENO-AO) presented in \cite{Balsara2016} since the linear weights exist at arbitrary points.\\
\ \\
The overarching idea of WENO-AO methods is to provide high order accuracy for smooth solutions over a large center stencil and adaptively reduce to lower order accuracy when the solution does not permit the high order accuracy. This is done by creating a nonlinear hybridization between a large center stencil with high order accuracy, and very stable lower order WENO schemes (e.g., CWENO schemes \cite{Levy1999}). Aside from the high order accuracy and existence of linear weights at arbitrary points, the robustness of these WENO-AO schemes is particularly attractive for our purposes. The authors in \cite{Balsara2016} write WENO-AO($p,r$) to denote an adaptive order that is at best $p$th order (from the large center stencil) and at worst $r$th order (from the stable lower order stencils). For our purposes we use WENO-AO(5,3). The end product of WENO and WENO-AO methods is ultimately a reconstruction polynomial $\mathcal{R}_j(x\in I_j)$ that we shall use for reconstruction.\\
\ \\
Equation \eqref{WENO} is valid when using WENO-AO reconstruction polynomials $\mathcal{R}_j(x\in I_j)$. In general, if the traceback points $\tilde{x}_{j-\frac{1}{2}}(t)$ and $\tilde{x}_{j+\frac{1}{2}}(t)$ reside in cells $I_{\ell}$ and $I_r$ respectively, then
\begin{equation}
    \label{WENOAO}
    \int_{\tilde{I}_j(t)}{u(x,t)dx} \approx \int_{\tilde{I}_j(t)\cap I_{\ell}}{\mathcal{R}_{\ell}(x)dx} + \Delta x\sum\limits_{i=\ell+1}^{r-1}{\overline{u}_i(t)} + \int_{\tilde{I}_j(t)\cap I_{r}}{\mathcal{R}_{r}(x)dx}.
\end{equation}
Below, we summarize the integral reconstruction procedure using WENO-AO as the Algorithm 1.\\

\noindent\xdashthick\\
\textbf{Algorithm 1.} Integral reconstruction using WENO-AO\\
\xdashthin\\
\textbf{Input:} uniform cell averages $\overline{u}_j(t)$ on the background grid of nodes $x_{j+\frac{1}{2}}$.\\
\textbf{Output:} possibly nonuniform cell averages $\tilde{u}_j(t)$ on the traceback grid of nodes $\tilde{x}_{j+\frac{1}{2}}(t)$.\\
\par\qquad \textbf{do $j=1,2,...,N$}
\par\qquad\qquad Locate the uniform background cells that $\tilde{x}_{j-\frac{1}{2}}(t)$ and $\tilde{x}_{j+\frac{1}{2}}(t)$ reside in.
\par\qquad\qquad Call these cells $I_{\ell}$ and $I_r$, respectively.
\par\qquad\qquad\textbf{if $\ell=r$}
\par\qquad\qquad\qquad Compute $\mathlarger{\tilde{u}_j(t) \approx \frac{1}{\Delta\tilde{x}_j(t)}\int_{\tilde{I}_j(t)}{\mathcal{R}_{\ell}(x)dx}}$.
\par\qquad\qquad\textbf{else}
\par\qquad\qquad\qquad\textbf{do $k=\ell,r$}
\par\qquad\qquad\qquad\qquad Compute $\mathlarger{\int_{\tilde{I}_j(t)\cap I_k}{\mathcal{R}_k(x)dx}}$.
\par\qquad\qquad\qquad\textbf{end do}
\par\qquad\qquad\qquad Compute $\mathlarger{\tilde{u}_j(t) \approx \frac{1}{\Delta\tilde{x}_j(t)}\left(\int_{\tilde{I}_j(t)\cap I_{\ell}}{\mathcal{R}_{\ell}(x)dx} + \Delta x\sum\limits_{i=\ell+1}^{r-1}{\overline{u}_i(t)} + \int_{\tilde{I}_j(t)\cap I_{r}}{\mathcal{R}_{r}(x)dx}\right)}$.
\par\qquad\qquad\textbf{end if}
\par\qquad\textbf{end do}\\
\xdashthick

\subsection{Reconstruction of point values}
\label{sec2.3}

Referring back to Figure \ref{spacetime} and equations \eqref{semidiscrete0}-\eqref{semidiscrete2}, we also need to reconstruct the point values $u(\tilde{x}^-_{j+\frac{1}{2}}(t),t)$ and $u(\tilde{x}^+_{j+\frac{1}{2}}(t),t)$ for the modified flux function. However, the WENO-AO schemes in \cite{Balsara2016} assume a uniform grid for convenience and efficiency. We wish to avoid using nonuniform WENO methods since the linear weights need to be recomputed every step and will quickly become expensive. However, nonuniform WENO methods for two-dimensional problems \cite{Balsara2020,Zhu2017,Zhu2017b,Zhu2019,Zhu2018} might become a more reasonable and realistic choice when dealing with non-splitting algorithms. Yet, we still need to reconstruct the left and right limits on the generally nonuniform traceback grid. We propose using the continuous piecewise-linear transformation from the nonuniform grid consisting of nodes $\tilde{x}_{j+\frac{1}{2}}(t)$ to the uniform background grid; performing a WENO-AO reconstruction on the uniform grid; and mapping back to the nonuniform grid to obtain the desired limits. We call this the \textit{nonuniform-to-uniform transformation}. Given a fixed $t\in[t^n,t^{n+1}]$, consider the linear bijection $\phi_j:\tilde{I}_j(t)\rightarrow I_j$ for $j=1,2,...,N$. Letting $x\in\tilde{I}_j(t)$ and $\xi\in I_j$,
\begin{equation}
\label{eq: transform}
    \int_{\tilde{I}_j(t)}{u(x,t)dx} = |J|\int_{I_j}{u(x(\xi),t)d\xi},
\end{equation}
where $|J|(\xi) = |dx/d\xi|(\xi)$ is the Jacobian of the bijection $\phi_j$. In particular, the Jacobian is constant for the linear interpolant,
\begin{equation}
    \label{}
    |J|=\left|\frac{d}{d\xi}\left(\frac{\Delta\tilde{x}_j(t)}{\Delta x}\xi - \frac{x_{j+\frac{1}{2}}\tilde{x}_{j-\frac{1}{2}}(t) + x_{j-\frac{1}{2}}\tilde{x}_{j+\frac{1}{2}}(t)}{\Delta x}\right)\right| = \frac{\Delta\tilde{x}_j(t)}{\Delta x}.
\end{equation}
Since we want to apply WENO-AO over the uniform grid, we define the \textit{auxiliary uniform cell averages} as
\begin{equation}
    \check{u}_j(t) \coloneqq \frac{1}{\Delta x}\int_{I_j}{u(x(\xi),t)d\xi} = \frac{1}{\Delta x|J|}\int_{\tilde{I}_j(t)}{u(x,t)dx} = \frac{1}{\Delta\tilde{x}_j(t)}\int_{\tilde{I}_j(t)}{u(x,t)dx} = \tilde{u}_j(t).
\end{equation}
Note that we have shown the auxiliary uniform cell averages (on the uniform background grid) are identical to the nonuniform cell averages. Hence, under this continuous piecewise-linear mapping, we can directly use the nonuniform cell averages at time $t$ in the (uniform) WENO-AO procedure to obtain the left and right limits $u(\tilde{x}^-_{j+\frac{1}{2}}(t),t)$ and $u(\tilde{x}^+_{j+\frac{1}{2}}(t),t)$. We further note that the high-order accuracy is preserved with such a strategy only when there is a smooth mapping with equation \eqref{eq: transform}. Theoretical justification for the existence of such a mapping is highly nontrivial. Yet, our numerical tests verify that high-order spatial accuracy is achieved, when the characteristic field is smooth.\\

\noindent \xdashthick\\
\textbf{Algorithm 2.} Reconstruction using WENO-AO with the nonuniform-to-uniform transformation\\
\xdashthin\\
\textbf{Input:} nonuniform cell averages $\tilde{u}_j(t)$ on the traceback grid.\\
\textbf{Output:} left and right limits $u(\tilde{x}^-_{j+\frac{1}{2}}(t),t)$ and $u(\tilde{x}^+_{j+\frac{1}{2}}(t),t)$.\\
\par\qquad \textbf{do $j=1,2,...,N$}
\par\qquad\qquad Calculate the (uniform) WENO-AO reconstruction polynomial $\mathcal{R}_j(x\in I_j)$ in terms of
\par\qquad\qquad the neighboring auxiliary uniform cell averages $\check{u}_i(t)=\tilde{u}_i(t)$, $i=j-p,...,j+q$.
\par\qquad\qquad Compute the left limit $u(\tilde{x}^-_{j+\frac{1}{2}}(t),t) \approx \mathcal{R}_j(x_{j+\frac{1}{2}})$.
\par\qquad\qquad Compute the right limit $u(\tilde{x}^+_{j-\frac{1}{2}}(t),t) \approx \mathcal{R}_j(x_{j-\frac{1}{2}})$.
\par\qquad\textbf{end do}\\
\xdashthick

\subsection{Time evolution with explicit Runge-Kutta methods}
\label{sec2.4}

Algorithms 1 and 2 now allow us to perform (integral) reconstruction on a traceback grid consisting of nodes $\tilde{x}_{j+\frac{1}{2}}(t)$. With these two tools we can evolve equation \eqref{semidiscrete2} using any explicit RK method. Recall that our primary goal is to achieve high order accuracy while also taking large time steps. In our numerical tests we use WENO-AO(5,3) for the spatial approximation; although higher order WENO-AO methods can certainly be used. As such, we would like to use high-order time stepping methods. In the cases where the solution is smooth, the standard fourth-order RK method suffices. However, if the solution is discontinuous (e.g., a travelling Heaviside step function), then we require an explicit SSP RK method \cite{Gottlieb2001}. Explicit SSP RK methods are especially attractive when numerically solving hyperbolic conservation laws because they maintain strong stability while also achieving high order accuracy in time. When applicable, we use the optimal three-stage, third-order explicit SSP RK method outlined below for demonstrative purposes.

\begin{align}\label{SSPRK3}
\begin{split}
	\Delta\tilde{x}_j^{(0)}\tilde{u}_j^{(0)} &= \Delta x_j^*\overline{u}_j^*,\\
	\Delta\tilde{x}_j^{(1)}\tilde{u}_j^{(1)} &= \Delta\tilde{x}_j^{(0)}\tilde{u}_j^{(0)} - \Delta t\big(\hat{F}_{j+\frac{1}{2}}^{(0)}(u^-,u^+;t^n)-\hat{F}_{j-\frac{1}{2}}^{(0)}(u^-,u^+;t^n)\big),\\
	\Delta\tilde{x}_j^{(2)}\tilde{u}_j^{(2)} &= \Delta\tilde{x}_j^{(0)}\tilde{u}_j^{(0)} - \frac{\Delta t}{4}\Big[\big(\hat{F}_{j+\frac{1}{2}}^{(0)}(u^-,u^+;t^n)-\hat{F}_{j-\frac{1}{2}}^{(0)}(u^-,u^+;t^n)\big)\\
	&\qquad\qquad\qquad\qquad+ \big(\hat{F}_{j+\frac{1}{2}}^{(1)}(u^-,u^+;t^{n+1})-\hat{F}_{j-\frac{1}{2}}^{(1)}(u^-,u^+;t^{n+1})\big)\Big],\\
	\Delta x_j\overline{u}_j^{n+1} &= \Delta\tilde{x}_j^{(0)}\tilde{u}_j^{(0)} - \frac{2\Delta t}{3}\big(\hat{F}_{j+\frac{1}{2}}^{(2)}(u^-,u^+;t^{n+\frac{1}{2}}) - \hat{F}_{j-\frac{1}{2}}^{(2)}(u^-,u^+;t^{n+\frac{1}{2}})\big)\\
	&\qquad\qquad\qquad\qquad- \frac{\Delta t}{6}\Big[\big(\hat{F}_{j+\frac{1}{2}}^{(0)}(u^-,u^+;t^n)-\hat{F}_{j-\frac{1}{2}}^{(0)}(u^-,u^+;t^n)\big)\\
	&\qquad\qquad\qquad\qquad+ \big(\hat{F}_{j+\frac{1}{2}}^{(1)}(u^-,u^+;t^{n+1})-\hat{F}_{j-\frac{1}{2}}^{(1)}(u^-,u^+;t^{n+1})\big)\Big],
\end{split}
\end{align}

\noindent where $\hat{F}_{j\pm\frac{1}{2}}^{(k)}(u^-,u^+;t)$ denotes using the cell averages $\tilde{u}_j^{(k)}$ from stage $k$ to approximate the limits $u_{j\pm\frac{1}{2}}^{-}$ and $u_{j\pm\frac{1}{2}}^{+}$ in the numerical flux function at time $t$ using Algorithms 1 and 2. The Butcher tables for other explicit RK methods are provided in Appendix A.

\begin{remark}
If the approximate characteristics are defined such that $\nu_{j+\frac{1}{2}}=0$ for all $j$, then the EL-RK-FV scheme reduces to the standard RK-FV scheme \cite{Shu2009}. Referring to Figure \ref{spacetime}, the approximate characteristics would be vertical lines.
\end{remark}

\subsection{Two-dimensional problems}
\label{sec2.5}

We now consider the two-dimensional equation
\begin{equation}
\label{2Dscalar}
    u_t + f(u)_x + g(u)_y = 0.
\end{equation}
The spatial domain is discretized into $N_xN_y$ cells, $I_{i,j}\coloneqq I_i\times I_j$ with centers $x_{i,j}=((x_{i-\frac{1}{2}}+x_{i+\frac{1}{2}})/2,(y_{j-\frac{1}{2}}+y_{j+\frac{1}{2}})/2)$, where the spatial discretizations in $x$ and $y$ are respectively given by
\[0=x_{\frac{1}{2}}<x_{\frac{3}{2}}<...<x_{N_x-\frac{1}{2}}<x_{N_x+\frac{1}{2}}=2\pi\qquad\text{and}\qquad 0=y_{\frac{1}{2}}<y_{\frac{3}{2}}<...<y_{N_y-\frac{1}{2}}<y_{N_y+\frac{1}{2}}=2\pi.\]
The CFL number is defined as
\begin{equation}
    \label{CFL2D}
    \Delta t = \frac{CFL}{\mathlarger{\frac{\max{|f'(u)|}}{\Delta x} + \frac{\max{|g'(u)|}}{\Delta y}}},
\end{equation}
and the uniform cell averages at time $t^n$ are defined by
\begin{equation}
    \tilde{\bar{u}}_{i,j}^n\coloneqq \frac{1}{\Delta x\Delta y}\int_{I_{i,j}}{u(x,y,t^n)dxdy}.
\end{equation}
In the two-dimensional case, we also want to compute the \textit{interval averages} over the interval $I_i$ at a fixed $y$, or over the interval $I_j$ at a fixed $x$. Define the uniform \textit{interval averages} at time $t^n$ with one variable fixed by
\begin{subequations}
\begin{equation}
    \bar{u}_{i|y}^n\coloneqq \frac{1}{\Delta x}\int_{I_i}{u(x,y,t^n)dx},
\end{equation}
\begin{equation}
    \tilde{u}_{j|x}^n\coloneqq \frac{1}{\Delta y}\int_{I_j}{u(x,y,t^n)dy}.
\end{equation}
\end{subequations}
Note that \textit{only in this subsection} will the superscript tilde denote uniform interval averages in $y$, not nonuniform interval averages.\\
\ \\
Dimensional splitting methods are commonly used to solve two (or higher) dimensional problems, with second-order Strang splitting being a popular choice. Dimensional splitting methods solve equation \eqref{2Dscalar} by alternating between solving the easier problems
\begin{subequations}
\begin{equation}\label{Strang1}
	u_t + f(u)_x = 0,
\end{equation}
\begin{equation}\label{Strang2}
	u_t + g(u)_y = 0.
\end{equation}
\end{subequations}
Strang splitting uses $\tilde{\bar{u}}_{i,j}^n$ to solve \eqref{Strang1} over a half time step $\Delta t/2$ for intermediate solution $\tilde{\bar{u}}_{i,j}^*$; uses $\tilde{\bar{u}}_{i,j}^*$ to solve \eqref{Strang2}  over a full time step $\Delta t$ for intermediate solution $\tilde{\bar{u}}_{i,j}^{**}$; and uses $\tilde{\bar{u}}_{i,j}^{**}$ to solve \eqref{Strang1} over another half time step $\Delta t/2$ for solution $\tilde{\bar{u}}_{i,j}^{n+1}$. As such, we can reduce the two-dimensional problem to solving several (easier) one-dimensional problems. Since the one-dimensional EL-RK-FV algorithm evolves interval averages, we need a way to go between two-dimensional cell averages and one-dimensional interval averages. We note that the nonlinear weights in the WENO-AO method \cite{Balsara2016} are the same for all nodes within the same cell.

\subsubsection{Going from/to cell averages to/from interval averages}

Consider the interval averages as functions of $y$ and $x$, respectively defined by
\begin{subequations}
\begin{equation}
    \psi_i(y) \coloneqq \bar{u}_{i|y}^n = \frac{1}{\Delta x}\int_{I_i}{u(x,y,t^n)dx},
\end{equation}
\begin{equation}
    \phi_j(x) \coloneqq \tilde{u}_{j|x}^n = \frac{1}{\Delta y}\int_{I_j}{u(x,y,t^n)dy}.
\end{equation}
\end{subequations}
 For any given $i=1,2,...,N_x$ or $j=1,2,...,N_y$, consider the respective Gauss-Legendre quadrature nodes $\{x_k^{(i)}\in I_i\text{ : }k=1,...,K\}$ and $\{y_l^{(j)}\in I_j\text{ : }l=1,...,L\}$. Observing that the cell averages at time $t^n$ can be expressed as the interval averages of $\psi(y)$ and $\phi(x)$,
 \begin{equation}
     \tilde{\bar{u}}_{i,j}^n = \frac{1}{\Delta y}\int_{I_j}{\psi_i(y)dy} = \frac{1}{\Delta x}\int_{I_i}{\phi_j(x)dx},
 \end{equation}
we can use WENO-AO to go from cell averages to interval averages evaluated at the Gauss-Legendre nodes,
\begin{subequations}
\begin{equation}
	\tilde{\bar{u}}_{i,j}^n\longrightarrow\psi_i(y_l^{(j)})=\bar{u}_{i|l}^n,
\end{equation}
\begin{equation}
	\tilde{\bar{u}}_{i,j}^n\longrightarrow\phi_j(x_k^{(i)})=\tilde{u}_{j|k}^n.
\end{equation}
\end{subequations}

\noindent Algorithm 3 presents how to implement WENO-AO to go from cell averages to $x-$interval averages at the fixed $y-$Gauss-Legendre nodes. A similar algorithm can be done to go from cell averages to $y-$interval averages at the fixed $x-$Gauss-Legendre nodes.\\

\noindent\xdashthick\\
\textbf{Algorithm 3.} Going from cell averages to $x-$interval averages\\
\xdashthin\\
\textbf{Input:} uniform cell averages $\tilde{\bar{u}}_{i,j}$.\\
\textbf{Output:} uniform $x-$interval averages $\bar{u}_{i|l}$ at fixed Gauss-Legendre nodes $\{y_l^{(j)}\in I_j\text{ : }l=1,...,L\}$.\\
\par\qquad\textbf{do $i=1,2,...,N_x$}
\par\qquad\qquad \textbf{do $j=1,2,...,N_y$}
\par\qquad\qquad\qquad Calculate the WENO-AO reconstruction polynomial $\mathcal{R}_j^{(i)}(y\in I_j)$ in terms of the
\par\qquad\qquad\qquad neighboring averages $\tilde{\bar{u}}_{i,k}$, $k=j-p,...,j+q$.
\par\qquad\qquad\qquad\textbf{do $l=1,...,L$}
\par\qquad\qquad\qquad\qquad Store $\bar{u}_{i|l} = \psi_i(y_l^{(j)}) \approx \mathcal{R}_j^{(i)}(y_l^{(j)})$.
\par\qquad\qquad\qquad\textbf{end do}
\par\qquad\qquad\textbf{end do}
\par\qquad\textbf{end do}
\par\qquad The uniform $x-$interval averages at a fixed Gauss-Legendre node $y_l^{(j)}$ are $\{\bar{u}_{i|l} = \psi_i(y_l^{(j)})\text{ : }l=1,2,...,N_x\}$.\\
\xdashthick
\\
\ \\
Computing the cell averages from the interval averages at the Gauss-Legendre nodes is straightforward using a Gaussian quadrature. Let $\xi$ and $w$ denote the standard Gauss-Legendre nodes and weights over the interval $[-1,1]$, respectively. Without loss of generality, we can go from $x-$interval averages to cell averages.
\begin{align}
\begin{split}
    \label{intervaltocell1}
	\tilde{\bar{u}}_{i,j} &= \frac{1}{\Delta y}\int_{I_j}{\psi(y)dy}\\
	&= \frac{1}{2}\int_{-1}^{1}{\psi\Big(y_{j-\frac{1}{2}}+\frac{\Delta y}{2}(y'+1)\Big)dy'}\\
	&\approx \frac{1}{2}\sum\limits_{l=1}^{L}{w_l\psi\Big(y_{j-\frac{1}{2}}+\frac{\Delta y}{2}(\xi_l+1)\Big)}\\
	&= \frac{1}{2}\sum\limits_{l=1}^{L}{w_l\psi(y_l^{(j)})},\qquad\text{where $y_l^{(j)}=y_{j-\frac{1}{2}}+\frac{\Delta y}{2}(\xi_l+1)$.}
\end{split}
\end{align}

\subsubsection{Strang splitting}

For demonstrative purposes, we outline the EL-RK-FV algorithm for two-dimensional problems via Strang splitting. Since Strang splitting is only second-order in time, higher-order splitting methods might be preferred. We present the fourth-order splitting method developed by Forest and Ruth \cite{Forest1990} and Yoshida \cite{Yoshida1990} in Appendix B; Rossmanith and Seal used this fourth-order splitting method in the semi-Lagrangian framework in \cite{Rossmanith2011}.\\

\noindent\xdashthick\\
\textbf{Algorithm 4.} Strang splitting for the EL-RK-FV method\\
\xdashthin\\
\textbf{Input:} uniform cell averages $\tilde{\bar{u}}_{i,j}^n$.\\
\textbf{Output:} uniform cell averages $\tilde{\bar{u}}_{i,j}^{n+1}$.\\
\par\qquad\textbf{Step 1 ($x$-direction).} Solve equation \eqref{Strang1} for a half time step $\Delta t/2$; that is, over $[t^n,t^{n+1/2}]$.
\par\qquad Use Algorithm 3 to get $x-$interval averages; that is, $\tilde{\bar{u}}_{i,j}^n\longrightarrow\bar{u}_{i|l}^n$.
\par\qquad\textbf{do $l=1,2,...,N_y\cdot L$}
\par\qquad\qquad For each Gauss-Legendre node, update the $x-$interval averages by applying the 1D algorithm;
\par\qquad\qquad that is, $\bar{u}_{i|l}^n\longrightarrow\bar{u}_{i|l}^*$.
\par\qquad\textbf{end do}
\par\qquad Use equation \eqref{intervaltocell1} to get updated cell averages; that is, $\bar{u}_{i|l}^*\longrightarrow\tilde{\bar{u}}_{i,j}^*$.\\

\par\qquad\textbf{Step 2 ($y$-direction).} Solve equation \eqref{Strang2} for a full time step $\Delta t$; that is, over $[t^n,t^{n+1}]$.
\par\qquad Use Algorithm 3 analogue to get $y-$interval averages; that is, $\tilde{\bar{u}}_{i,j}^*\longrightarrow\tilde{u}_{j|k}^*$.
\par\qquad\textbf{do $k=1,2,...,N_x\cdot K$}
\par\qquad\qquad For each Gauss-Legendre node, update the $y-$interval averages by applying the 1D algorithm;
\par\qquad\qquad that is, $\tilde{u}_{j|k}^{*}\longrightarrow\tilde{u}_{j|k}^{**}$.
\par\qquad\textbf{end do}
\par\qquad Use equation \eqref{intervaltocell1} analogue to get updated cell averages; that is, $\tilde{u}_{j|k}^{**}\longrightarrow\tilde{\bar{u}}_{i,j}^{**}$.\\

\par\qquad\textbf{Step 3 ($x$-direction).} Solve equation \eqref{Strang1} for a half time step $\Delta t/2$; that is, over $[t^{n+1/2},t^{n+1}]$.
\par\qquad Use Algorithm 3 to get $x-$interval averages; that is, $\tilde{\bar{u}}_{i,j}^{**}\longrightarrow\bar{u}_{i|l}^{**}$.
\par\qquad\textbf{do $l=1,2,...,N_y\cdot L$}
\par\qquad\qquad For each Gauss-Legendre node, update the $x-$interval averages by applying the 1D algorithm;
\par\qquad\qquad that is, $\bar{u}_{i|l}^{**}\longrightarrow\bar{u}_{i|l}^{n+1}$.
\par\qquad\textbf{end do}
\par\qquad Use equation \eqref{intervaltocell1} to get updated cell averages; that is, $\bar{u}_{i|l}^{n+1}\longrightarrow\tilde{\bar{u}}_{i,j}^{n+1}$.
\\
\xdashthick


\section{The EL-RK-FV method for convection-diffusion equations}

Throughout this section, overlines (e.g., $\overline{u}_j(t)$) denote \textit{uniform} cell averages, and tildes (e.g., $\tilde{u}_j(t)$) denote \textit{nonuniform} cell averages. We discuss the EL-RK-FV algorithm for the convection-diffusion equation \eqref{CDeqn}. Since we use dimensional splitting methods for two-dimensional problems, it suffices to discuss the 1D EL-RK-FV algorithm for problems of the form
\begin{equation}
    \label{1Dconvectiondiffusion}
    u_t + f(u)_x = \epsilon u_{xx} + g(x,t),
\end{equation}
where we impose periodic boundary conditions. Rewriting equation \eqref{1Dconvectiondiffusion} in divergence form, integrating over the same space-time region $\Omega_j$ as in the $\epsilon=0$ case, applying the divergence theorem and fundamental theorem of calculus, integrating over $[t^n,t^{n+1}]$, and converting to the time-differential form, we get the semi-discrete formulation
\begin{equation}
    \label{semidiscrete3}
    \frac{d}{dt}\int_{\tilde{I}_j(t)}{u(x,t)dx} = -\left[\hat{F}_{j+\frac{1}{2}}(t)-\hat{F}_{j-\frac{1}{2}}(t)\right] + \epsilon\int_{\tilde{I}_j(t)}{u_{xx}(x,t)dx} + \int_{\tilde{I}_j(t)}{g(x,t)dx},
\end{equation}
where $\hat{F}_{j+\frac{1}{2}}(t)$ is any monotone numerical flux (e.g., Lax-Friedrichs flux) and $F_{j+\frac{1}{2}}(t)$ is the modified flux function defined in \eqref{semidiscrete1}.

\subsection{Computing the uniform cell averages of $u_{xx}$}

When evolving equation \eqref{semidiscrete3} we will need to compute the uniform cell averages of $u_{xx}(x,t)$, defined by
\begin{equation}
    \overline{u}_{xx,j}(t)\coloneqq \frac{1}{\Delta x}\int_{I_j}{u_{xx}(x,t)dx}.
\end{equation}
We use the centered five-point stencil $\{j-2,j-1,j,j+1,j+2\}$ for linear reconstruction. Let $P(x\in I_j)$ be the linear reconstruction polynomial over the centered five-point stencil. After some tedious algebra, we find that the uniform cell averages $\overline{u}_{xx,j}(t)$ can be expressed in terms of the uniform cell averages $\overline{u}_j(t)$ with fourth-order accuracy using the equation
\begin{equation}
    \label{5stencil}
    \overline{u}_{xx,j}(t) = \frac{1}{\Delta x^2}
    \begin{bmatrix*}
        -\frac{1}{12}&\frac{4}{3}&-\frac{5}{2}&\frac{4}{3}&-\frac{1}{12}
    \end{bmatrix*}
    \begin{bmatrix*}
        \overline{u}_{j-2}(t)\\
        \overline{u}_{j-1}(t)\\
        \overline{u}_{j}(t)\\
        \overline{u}_{j+1}(t)\\
        \overline{u}_{j+2}(t)
    \end{bmatrix*}.
\end{equation}

\noindent For implementation it is easier to express equation \eqref{5stencil} in matrix form (assuming periodic or zero boundary conditions). We denote this matrix $\mathbf{D}_4$ in Algorithm 5. Computing the nonuniform cell averages $\tilde{u}_{xx,j}(t)$ can now be done using Algorithm 1.

\subsection{Time evolution with implicit-explicit Runge-Kutta methods}

We can split equation \eqref{1Dconvectiondiffusion} into its non-stiff and stiff terms,
\begin{equation}
    \label{FG}
    u_t = \mathcal{F}(u;x,t) + \mathcal{G}(u;x,t),
\end{equation}
where $\mathcal{F}(u;x,t)=-f(u)_x$ is the non-stiff convective term, and $\mathcal{G}(u;x,t)=\epsilon u_{xx}+g(x,t)$ is the stiff diffusive (and source) term. Discretization methods used for such problems are implicit-explicit (IMEX) RK schemes \cite{Ascher1997,Conde2017,Higueras2014}. The intuition behind these schemes is straightforward -- evolve the non-stiff term explicitly, and evolve the stiff term implicitly. As such, each stage in the RK method will involve explicitly evaluating the non-stiff term, and solving a linear system due to the stiff term.\\
\ \\
All that's needed are the possibly nonuniform cell averages at each RK stage $\mu=1,...,s$ over the space-time regions $\Omega_j$. There are two approaches one can take to approximate the possibly nonuniform cell averages at each RK stage: (1) have a single space-time partition for multiple RK stages and directly update the solution to the possibly nonuniform traceback grid at each intermediate RK stage, or (2) have multiple space-time partitions for intermediate RK stages, and compute the uniform cell averages at time $t^{(\mu)}$ and project them onto the possibly nonuniform traceback grid via Algorithm 1. Although the first approach is more intuitive, it is computationally expensive since the system to be solved (from the implicit part) will depend on the measure of each cell. We choose the second approach, as computing the uniform cell averages at time $t^{(\mu)}$ requires solving a system dependent only on the background uniform mesh size $\Delta x$. Since we choose the second approach, we will need to: (1) solve for the \textit{uniform} cell averages at each consecutive stage, and (2) project the uniform cell averages onto the possibly nonuniform traceback grid via Algorithm 1. Recall the uniform cell averages $\overline{u}_{xx,j}^{(\mu)}$ can be computed from the uniform cell averages $\overline{u}_{j}^{(\mu)}$ using equation \eqref{5stencil}.\\
\ \\
For simplicity, we only use the IMEX RK schemes outlined in \cite{Ascher1997}. As per the notation used by Ascher, et al., \textit{IMEX($s$,$\sigma$,$p$)} denotes using an $s-$stage diagonally-implicit Runge-Kutta (DIRK) scheme for $\mathcal{G}(u;x,t)$, using a $\sigma-$stage explicit RK scheme for $\mathcal{F}(u;x,t)$, and being of combined order $p$. Consider the semi-discrete formulation \eqref{semidiscrete3} rewritten as
\begin{equation}
    \frac{d}{dt}\int_{\tilde{I}_j(t)}{u(x,t)dx} = \mathcal{F}(u;t) + \mathcal{G}(u;t),
\end{equation}
where we redefine $\mathcal{F}(u;t)=-\left[\hat{F}_{j+\frac{1}{2}}(t)-\hat{F}_{j-\frac{1}{2}}(t)\right]$ and $\mathcal{G}(u;t) =  \epsilon\int_{\tilde{I}_j(t)}{u_{xx}(x,t)dx} + \int_{\tilde{I}_j(t)}{g(x,t)dx}$. The IMEX($s$,$\sigma$,$p$) RK schemes are expressed with two Butcher tables: one for the implicit RK method, and another for the explicit RK method.
\begin{table}[h!]
    \centering
    \begin{minipage}[b]{0.49\linewidth}
    \centering
    \caption*{Implicit Scheme}
    \begin{tabular}{c|lllll}
        0&0&0&0&$\hdots$&0\\
        $c_1$&0&$a_{11}$&0&$\hdots$&0\\
        $c_2$&0&$a_{21}$&$a_{22}$&$\hdots$&0\\
        $\vdots$&$\vdots$&$\vdots$&$\vdots$&$\ddots$&$\vdots$\\
        $c_s$&0&$a_{s1}$&$a_{s2}$&$\hdots$&$a_{ss}$\\
        \hline
        &0&$b_1$&$b_2$&$\hdots$&$b_s$
        \end{tabular}
    \end{minipage}
    \begin{minipage}[b]{0.49\linewidth}
    \centering
    \caption*{Explicit Scheme}
    \begin{tabular}{c|lllll}
        0&0&0&0&$\hdots$&0\\
        $c_1$&$\hat{a}_{21}$&0&0&$\hdots$&0\\
        $c_2$&$\hat{a}_{31}$&$\hat{a}_{32}$&0&$\hdots$&0\\
        $\vdots$&$\vdots$&$\vdots$&$\vdots$&$\ddots$&$\vdots$\\
        $c_s$&$\hat{a}_{\sigma 1}$&$\hat{a}_{\sigma 2}$&$\hat{a}_{\sigma 3}$&$\hdots$&0\\
        \hline
        &$\hat{b}_1$&$\hat{b}_2$&$\hat{b}_3$&$\hdots$&$\hat{b}_{\sigma}$
        \end{tabular}
    \end{minipage}
\end{table}

\noindent Defining $U^{(\mu)}\coloneqq\int_{\tilde{I}_j(t^{(\mu)})}{u(x,t^{(\mu)})dx}$, the IMEX($s$,$\sigma$,$p$) scheme over the space-time region $\Omega_j$ is as follows:
\begin{subequations}
\begin{align}
\begin{split}
\label{IMEX}
    U^{n+1} &= U^n + \Delta t\sum\limits_{\mu=1}^{s}{b_{\mu}K_{\mu}} + \Delta t\sum\limits_{\mu=1}^{\sigma}{\hat{b}_{\mu}\hat{K}_{\mu}},
\end{split}\\
\begin{split}
\label{K}
    K_{\mu} &= \mathcal{G}(U^{(\mu)};t^{(\mu)}),\qquad \mu=1,2,...,s,
\end{split}\\
\begin{split}
\label{Khat1}
    \hat{K}_1 &= \mathcal{F}(U^n;t^n),
\end{split}\\
\begin{split}
\label{Khat}
    \hat{K}_{\mu+1} &= \mathcal{F}(U^{(\mu)};t^{(\mu)}),\qquad \mu=1,2,...,s.
\end{split}
\end{align}
\end{subequations}

\noindent More precisely,
\begin{subequations}
\begin{align}
\begin{split}
    K_{\mu} &= \epsilon\int_{\tilde{I}_j(t^{(\mu)})}{u_{xx}(x,t^{(\mu)})dx} + \int_{\tilde{I}_j(t^{(\mu)})}{g(x,t^{(\mu)})dx},
\end{split}\\
\begin{split}
    \hat{K}_{\mu+1} &= -\left[\hat{F}_{j+\frac{1}{2}}(t^{(\mu)}) - \hat{F}_{j-\frac{1}{2}}(t^{(\mu)})\right].
\end{split}
\end{align}
\end{subequations}
Based on the IMEX RK method, the solution $U^{(\mu)}$ over the traceback grid can be approximated by
\begin{equation}
\label{Um}
    U^{(\mu)} = U^{n} + \Delta t\sum\limits_{\nu=1}^{\mu}{a_{\mu,\nu}K_{\nu}} + \Delta t\sum\limits_{\nu=1}^{\mu}{\hat{a}_{\mu+1,\nu}\hat{K}_{\nu}},\qquad \mu=1,2,...,s.
\end{equation}
Recall that we choose \textit{not} to directly update the solution over $\Omega_j$. Instead, we opt to solve for the uniform cell averages at each RK stage and project them onto the corresponding possibly nonuniform traceback grid. We define sub-space-time regions ${}_{\mu}\Omega_j$ for the $\mu$-th stage of the IMEX RK scheme, as shown in Figures \ref{spacetime1} and \ref{spacetime2}. The space-time region ${}_{\mu}\Omega_j$ traces back to time $t^n$ (using the same approximate characteristic speeds as in $\Omega_j$) from time $t^{(\mu)}$. Hence, at time $t^{(\mu)}$ on the sub-space-time region ${}_{\mu}\Omega_j$ the grid is uniform. Lower left subscript $\mu$ denotes values confined to the sub-space-time region ${}_{\mu}\Omega_j$. For example, ${}_{2}U^n$ are the possibly nonuniform definite integrals (and hence the possibly nonuniform cell averages) at time $t^n$ over the traceback cell in ${}_{2}\Omega_j$, as seen in Figure \ref{spacetime2}.

\begin{figure}[h!]
\begin{minipage}[b]{0.49\linewidth}
    \centering
	\includegraphics[width=\textwidth]{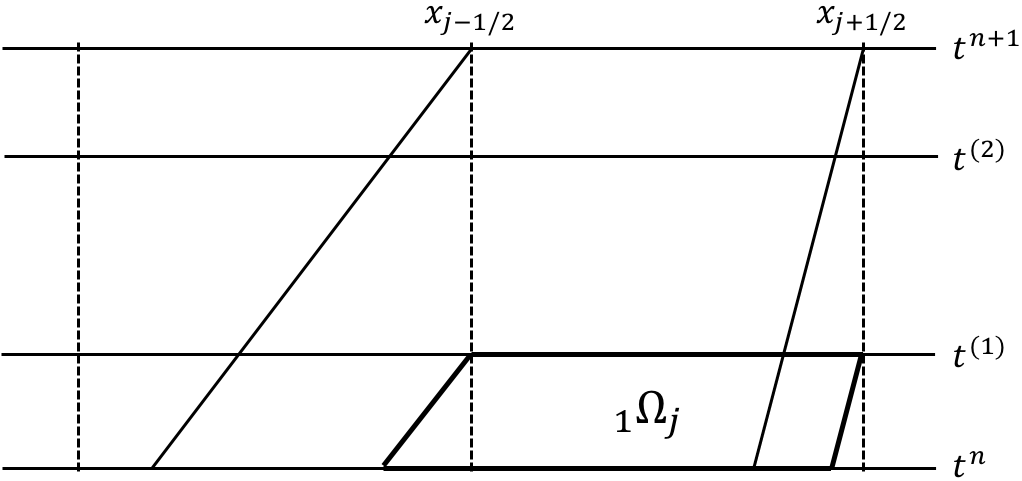}
	\caption{The space-time region ${}_{1}\Omega_j$.}
	\label{spacetime1}
\end{minipage}
\begin{minipage}[b]{0.49\linewidth}
    \centering
	\includegraphics[width=\textwidth]{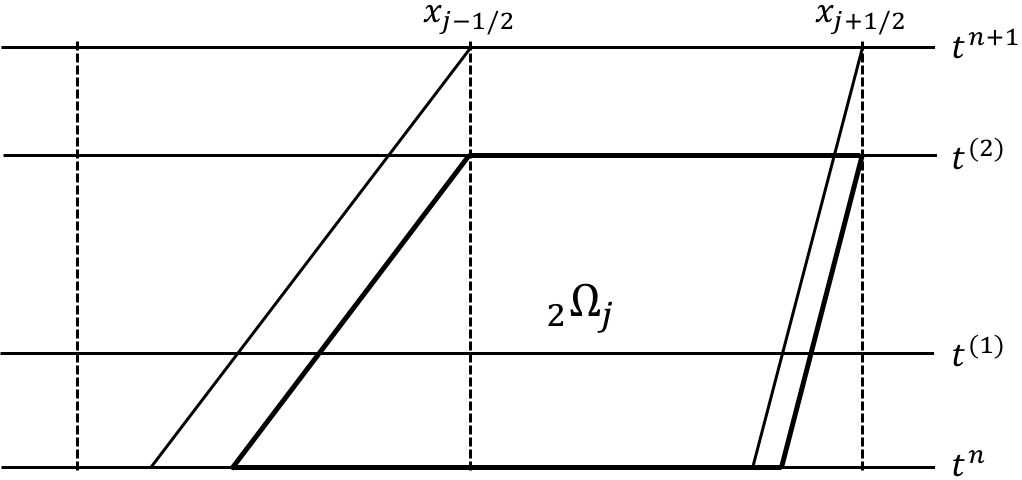}
	\caption{The space-time region ${}_{2}\Omega_j$.}
	\label{spacetime2}
\end{minipage}
\end{figure}

\noindent The desired uniform cell averages $\overline{u}_j^{(\mu)} = {}_{\mu}U^{(\mu)}/\Delta x$ at stage $\mu$ can be obtained by computing the first $\mu-$stages of the IMEX RK scheme over the sub-space-time region ${}_{\mu}\Omega_j$, given below.
\begin{subequations}
\label{IMEXsubstage}
\begin{align}
\begin{split}
    {}_{\mu}U^{(\mu)}& = {}_{\mu}U^n + \Delta t\sum\limits_{\nu=1}^{\mu}{(a_{\mu,\nu})({}_{\mu}K_{\nu})} + \Delta t\sum\limits_{\nu=1}^{\mu}{(\hat{a}_{\mu+1,\nu})({}_{\mu}\hat{K}_{\nu})}
\end{split}\\
\begin{split}
    {}_{\mu}K_{\nu} &= \mathcal{G}({}_{\mu}U^{(\nu)};t^{(\nu)}),\qquad \nu=1,2,...,\mu,
\end{split}\\
\begin{split}
    {}_{\mu}\hat{K}_1 &= \mathcal{F}({}_{\mu}U^n;t^n),
\end{split}\\
\begin{split}
 {}_{\mu}\hat{K}_{\nu+1} &= \mathcal{F}({}_{\mu}U^{(\nu)};t^{(\nu)}),\qquad \nu=1,2,...,\mu.
\end{split}
\end{align}
\end{subequations}
\noindent More precisely,
\begin{subequations}
\begin{align}
\begin{split}
    {}_{\mu}K_1 &= \epsilon\int_{{}_{\mu}\tilde{I}_j(t^{(\nu)})}{u_{xx}(x,t^{(\nu)})dx} + \int_{{}_{\mu}\tilde{I}_j(t^{(\nu)})}{g(x,t^{(\nu)})dx},
\end{split}\\
\begin{split}
    {}_{\mu}\hat{K}_{\nu+1} &= -\left[\hat{F}_{j+\frac{1}{2}}(t^{(\nu)}) - \hat{F}_{j-\frac{1}{2}}(t^{(\nu)})\right].
\end{split}
\end{align}
\end{subequations}
Note that ${}_{\mu}\hat{K}_{\nu+1}$ uses the cell averages restricted to the sub-space-time regions ${}_{\mu}\Omega_j$. Further note that equation \eqref{IMEXsubstage} can recycle and reuse the uniform cell averages already computed during stages $1,2,...,\mu-1$. Simply project the uniform cell averages ($\overline{u}_j^{(\nu)}$ or $\overline{u}_{xx,j}^{(\nu)}$, $\nu=1,2,...,\mu-1$) onto the traceback grids formed by the space-time regions ${}_{\mu}\Omega_{j}$. Although IMEX RK schemes are straightforward, it is easy to get lost in the notation. To help demonstrate the EL-RK-FV algorithm coupled with an IMEX RK scheme, we present the IMEX(2,2,2) case in Appendix D. The Butcher tables for other IMEX RK schemes are included in Appendix C. There are IMEX SSP RK schemes that can be found in \cite{Conde2017,Higueras2014}.\\
\clearpage
\noindent\xdashthick\\
\textbf{Algorithm 5.} EL-RK-FV algorithm coupled with an IMEX RK scheme \cite{Ascher1997}, $IMEX(s,\sigma,p)$\\
\xdashthin\\
\textbf{Input:} uniform cell averages $\overline{u}_j^n$.\\
\textbf{Output:} uniform cell averages $\overline{u}_j^{n+1}$.\\
\par\qquad Compute the possibly nonuniform traceback cell averages $\tilde{u}_{j}(t^n)$ using Algorithm 1.
\par\qquad Store $\hat{K}_1 = \mathcal{F}(U^n;t^n)$. This is the traceback grid in $\Omega_j$.
\par\qquad\textbf{do $\mu=1,2,...,s$}
\par\qquad\qquad 1. Compute the uniform cell averages $\overline{u}_j^{(\mu)}$ at time $t^{(\mu)}$ with equation \eqref{IMEXsubstage} by:
\par\qquad\qquad\qquad i. Restrict yourself to the space-time region ${}_{\mu}\Omega_j$.
\par\qquad\qquad\qquad ii. Compute the values ${}_{\mu}K_{\nu}$ and ${}_{\mu}\hat{K}_{\nu+1}$ at each stage $\nu=1,2,...,\mu-1$ using Algorithm 1
\par\qquad\qquad\qquad to compute the necessary nonuniform cell averages, and Algorithm 2 for the modified flux
\par\qquad\qquad\qquad function. Definite integrals of $g(x,t)$ can be computed using a Gaussian quadrature.
\par\qquad\qquad\qquad iii. Recalling equation \eqref{5stencil}, solve equation \eqref{IMEXsubstage} by setting it up as the linear system
\[\left(\mathbf{I} - \frac{a_{\mu,\mu}\epsilon\Delta t}{\Delta x^2}\mathbf{D}_4\right){}_{\mu}\vv{U}^{(\mu)} = {}_{\mu}\vv{U}^n + \Delta t\sum\limits_{\nu=1}^{\mu-1}{(a_{\mu,\nu})({}_{\mu}\vv{K}_{\nu})} + \Delta t\sum\limits_{\nu=1}^{\mu}{(\hat{a}_{\mu+1,\nu})({}_{\mu}\vv{\hat{K}}_{\nu})} + a_{\mu,\mu}\Delta t\vv{g}(x,t^{(\mu)}),\]
\par\qquad\qquad\qquad where $\vv{g}_j(x,t^{(\mu)}) = \int_{I_j}{g(x,t^{(\mu)})dx}$.
\par\qquad\qquad\qquad iv. Store the uniform cell averages $\overline{u}_j^{(\mu)} = {}_{\mu}U_j^{(\mu)}/\Delta x$.
\par\qquad\qquad 2. Compute and store the uniform cell averages $\overline{u}_{xx,j}^{(\mu)}$ using equation \eqref{5stencil}.
\par\qquad\qquad 3. Compute the possibly nonuniform traceback cell averages $\tilde{u}_j^{(\mu)}$ using Algorithm 1.
\par\qquad\qquad 4. Compute the possibly nonuniform traceback cell averages $\tilde{u}_{xx,j}^{(\mu)}$ using Algorithm 1.
\par\qquad\qquad 5. Compute and store $K_{\mu}$ and $\hat{K}_{\mu+1}$. Note that we are back on the possibly nonuniform
\par\qquad\qquad traceback grid consisting of cells $\tilde{I}_j(t^{(\mu)})$.
\par\qquad\textbf{end do}\\
\par\qquad Compute $U^{n+1} = U^n + \Delta t\sum\limits_{\mu=1}^{s}{b_{\mu}K_{\mu}} + \Delta t\sum\limits_{\mu=1}^{\sigma}{\hat{b}_{\mu}\hat{K}_{\mu}}$.
\par\qquad Compute the uniform cell averages $\overline{u}_j^{n+1} = U_j^{n+1}/\Delta x$.\\
\xdashthick

\subsection{Mass conservation}

We now show that the 1D EL-RK-FV algorithm is mass conservative when coupled with any IMEX RK scheme in \cite{Ascher1997}. Since we extend the EL-RK-FV algorithm to higher dimensions via dimensional splitting in this paper, showing mass conservation for the 1D problem suffices. Note that mass conservation of the $\epsilon=0$ case can be proved just as easily.

\begin{prop}
(Mass conservative). The 1D EL-RK-FV algorithm coupled with any IMEX RK scheme in \cite{Ascher1997} for (non)linear convection-diffusion equations is mass conservative, assuming the source term $g(x,t)=0$ and periodic boundary conditions.
\end{prop}

\begin{proof}
Making use of the semi-discrete form of the convection-diffusion equation,
\begin{align}
\begin{split}
    \int_{I_j}{u(x,t^{n+1})dx} &= \int_{\tilde{I}_j(t^n)}{u(x,t^n)dx} - \int_{t^n}^{t^{n+1}}{\left\{\hat{F}_{j+\frac{1}{2}}(t)-\hat{F}_{j-\frac{1}{2}}(t) - \epsilon\int_{\tilde{I}_j(t)}{u_{xx}(x,t)dx}\right\}dt}\\
    &= \int_{\tilde{I}_j(t^n)}{u(x,t^n)dx} - \int_{t^n}^{t^{n+1}}{\left\{\hat{F}_{j+\frac{1}{2}}(t)-\hat{F}_{j-\frac{1}{2}}(t) - \epsilon\left(u_x(\tilde{x}_{j+\frac{1}{2}}(t),t) - u_x(\tilde{x}_{j-\frac{1}{2}}(t),t)\right)\right\}dt}.
\end{split}
\end{align}
Summing over all $j=1,2,...,N_x$ and making use of the periodic boundary conditions,
\begin{equation}
    \int_a^b{u(x,t^{n+1})dx} = \int_a^b{u(x,t^n)dx}.
\end{equation}
\end{proof}


\section{Numerical tests}

In this section, we present results applying the proposed EL-RK-FV algorithm to various benchmark problems. In particular, we include error tables and error plots to investigate the spatial and temporal convergence of the algorithm. Mass conservation is also numerically verified by applying the proposed algorithm to the 0D2V (zero dimensions in physical space and two dimensions in velocity space) Leonard-Bernstein Fokker-Planck equation. We assume a uniform mesh, apply WENO-AO(5,3) in Algorithms 1 and 2, use three Gauss-Legendre nodes in Algorithm 3, and use the fourth-order approximation given by equation \eqref{5stencil} for the diffusive term. Unless otherwise stated, for the time-stepping we use the fourth-order RK method for pure convection problems, and IMEX(2,3,3) for convection-diffusion equations. We also use second-order Strang splitting for two-dimensional convection-diffusion equations. Although higher-order splitting methods can be used for pure convection problems, it is well-known that negative time integration can lead to significant instabilities when dealing with a diffusive term.\\

\noindent There are three sources of error: spatial approximation, time-stepping, and splitting. Depending on the CFL number and test problem, these three sources of error will influence the observed order of convergence. We compute the $L^1$, $L^2$, and $L^{\infty}$ errors (in one-dimension),
\begin{subequations}
\begin{equation}
    \norm{\overline{u}-\overline{u}_{exact}}_1 = \Delta x\sum\limits_{j=1}^{N_x}{|\overline{u}_j-\overline{u}_{exact,j}|}
\end{equation}
\begin{equation}
    \norm{\overline{u}-\overline{u}_{exact}}_2 = \Delta x\sqrt{\sum\limits_{j=1}^{N_x}{|\overline{u}_j-\overline{u}_{exact,j}|^2}}
\end{equation}
\begin{equation}
    \norm{\overline{u}-\overline{u}_{exact}}_{\infty} = \max\limits_{1\leq j\leq N_x}{|\overline{u}_j-\overline{u}_{exact,j}|}
\end{equation}
\end{subequations}

\noindent Note that for the norms defined above, $\norm{\overline{u}-\overline{u}_{exact}}_1\leq  |\mathcal{D}|\norm{\overline{u}-\overline{u}_{exact}}_{\infty}$.


\subsection{Pure convection problems: one-dimensional tests}

\begin{example}(1D transport with constant coefficient)
\begin{equation}
    \label{1Dtransport_constant}
    u_t + u_x = 0,\qquad x\in[0,2\pi]
\end{equation}
with periodic boundary conditions and exact solution $u(x,t)=\sin{(x-t)}$. The errors provided in Table \ref{1Dtransport_constant_table} verify the convergence of the EL-RK-FV method when using WENO-AO(5,3) and forward Euler. As expected, we see fifth-order convergence despite the large CFL number. There is no temporal error for the convective part since the characteristics are traced exactly and hence $F_{j+\frac{1}{2}}(t) = 0$ for all $t\in[t^n,t^{n+1}]$ and $j=1,2,...,N_x$.
\end{example}

\begin{table}[h!] 
\begin{center}
\caption{Convergence study with spatial mesh refinement for equation \eqref{1Dtransport_constant} with forward Euler at $T_f=1$.}
\label{1Dtransport_constant_table}
\begin{tabular}{|*{7}{c|}} 
\hline 
\multicolumn{7}{|c|}{$CFL=8$}\\
\hline
$N_x$&$L^1$ Error&Order&$L^2$ Error&Order&$L^{\infty}$ Error&Order\\
\hline
50&1.09E-08&-&4.83E-09&-&2.86E-09&-\\
100&3.34E-10&5.03&1.48E-10&5.03&8.34E-11&5.10\\
200&9.86E-12&5.08&4.37E-12&5.08&2.51E-12&5.06\\
400&2.80E-13&5.14&1.43E-13&4.94&1.91E-13&3.72\\
\hline
\end{tabular} 
\end{center} 
\end{table}

\begin{example}(1D transport with variable coefficient in space)
\begin{equation}
    \label{1Dtransport_variablespace}
    u_t + (\sin{(x)}u)_x = 0,\qquad x\in[0,2\pi]
\end{equation}
with periodic boundary conditions and exact solution
\[u(x,t) = \frac{\sin{(2\arctan{(e^{-t}\tan{(x/2)})})}}{\sin{(x)}}.\]
As seen in Table \ref{1Dtransport_variablespace_table}, we observe the high-order convergence. As the CFL number (and hence the time step) increases, the temporal error starts to play more of a role, as evidenced by the fourth-order convergence. We verify the high-order temporal convergence in Figure \ref{1Dtransport_variablespace_plot} by fixing the mesh $N_x=400$ and varying the CFL from 0.2 to 20.
\end{example}

\begin{table}[h!] 
\begin{center}
\caption{Convergence study with spatial mesh refinement for equation \eqref{1Dtransport_variablespace} with RK4 at $T_f=1$.}
\label{1Dtransport_variablespace_table}
\begin{tabular}{|*{7}{c|}} 
\hline
\multicolumn{7}{|c|}{$CFL=0.3$}\\
\hline
$N_x$&$L^1$ Error&Order&$L^2$ Error&Order&$L^{\infty}$ Error&Order\\
\hline
50&2.76E-04&-&2.53E-04&-&3.42E-04&-\\
100&3.05E-06&6.50&2.53E-06&6.64&2.94E-06&6.87\\
200&9.78E-08&4.96&7.90E-08&5.00&9.86E-08&4.90\\
400&3.24E-09&4.91&2.59E-09&4.93&3.23E-09&4.93\\
\hline
\multicolumn{7}{|c|}{$CFL=8$}\\
\hline
$N_x$&$L^1$ Error&Order&$L^2$ Error&Order&$L^{\infty}$ Error&Order\\
\hline
50&6.81E-02&-&4.11E-02&-&4.51E-02&-\\
100&1.18E-03&5.85&6.23E-04&6.04&6.38E-04&6.15\\
200&5.89E-05&4.32&3.63E-05&4.10&5.46E-05&3.54\\
400&3.71E-06&3.99&2.57E-06&3.82&4.31E-06&3.66\\
\hline
\end{tabular} 
\end{center} 
\end{table}

\begin{example}(1D transport with variable coefficient in time)
\begin{equation}
    \label{1Dtransport_variabletime}
    u_t + \left(\frac{u}{t+1}\right)_x = 0,\qquad x\in[0,2\pi]
\end{equation}
with periodic boundary conditions and exact solution $u(x,t)=\text{exp}(-5(x-\log{(t+1)}-\pi)^2)$. Periodic boundary conditions are a valid assumption for sufficiently thin Gaussian curves and small enough times. The expected high-order convergence for both small and large CFL numbers is seen in Table \ref{1Dtransport_variabletime_table}. Fixing the mesh $N_x=400$ and varying the CFL number from 0.2 to 20, fifth-order convergence in time is seen in Figure \ref{1Dtransport_variabletime_plot}. Observe that there are two optimal CFL numbers for this mesh.
\end{example}

\begin{table}[h!] 
\begin{center}
\caption{Convergence study with spatial mesh refinement for equation \eqref{1Dtransport_variabletime} with RK4 at $T_f=1$.}
\label{1Dtransport_variabletime_table}
\begin{tabular}{|*{7}{c|}} 
\hline 
\multicolumn{7}{|c|}{$CFL=0.95$}\\
\hline
$N_x$&$L^1$ Error&Order&$L^2$ Error&Order&$L^{\infty}$ Error&Order\\
\hline
50&6.36E-03&-&5.80E-03&-&8.23E-03&-\\
100&2.37E-04&4.75&2.28E-04&4.67&5.33E-04&3.95\\
200&1.71E-06&7.12&1.30E-06&7.46&2.36E-06&7.82\\
400&4.40E-08&5.28&3.84E-08&5.08&6.02E-08&5.29\\
\hline
\multicolumn{7}{|c|}{$CFL=8$}\\
\hline
$N_x$&$L^1$ Error&Order&$L^2$ Error&Order&$L^{\infty}$ Error&Order\\
\hline
50&6.71E-02&-&5.72E-02&-&8.15E-02&-\\
100&7.74E-04&6.44&6.48E-04&6.46&1.06E-03&6.27\\
200&1.52E-05&5.67&1.26E-05&5.69&1.76E-05&5.91\\
400&9.65E-07&3.98&8.11E-07&3.96&9.61E-07&4.20\\
\hline
\end{tabular} 
\end{center} 
\end{table} 

\begin{figure}[h!]
\begin{minipage}{0.48\linewidth}
    \centering
    \includegraphics[width=\textwidth]{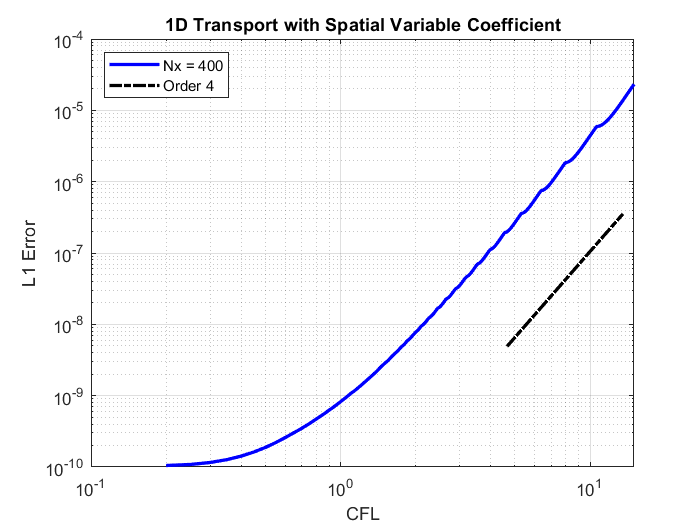}
    \caption{RK4, Final time $T_f=0.5$.}
    \label{1Dtransport_variablespace_plot}
\end{minipage}
\begin{minipage}{0.03\linewidth}
    \centering
    \ 
\end{minipage}
\begin{minipage}{0.48\linewidth}
    \centering
    \includegraphics[width=\textwidth]{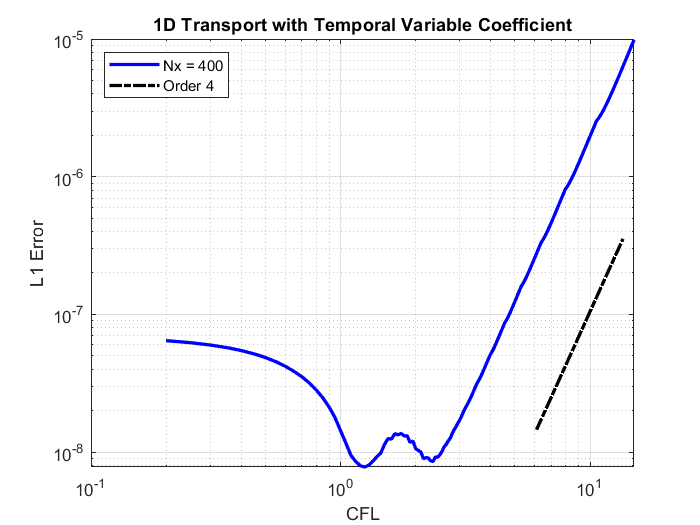}
    \caption{RK4, Final time $T_f=0.5$.}
    \label{1Dtransport_variabletime_plot}
\end{minipage}
\end{figure}


\subsection{Hyperbolic conservation laws: two-dimensional tests}

\begin{example}(2D transport with constant coefficient)
\begin{equation}
    \label{2Dtransport_constant}
    u_t + u_x + u_y = 0,\qquad x,y\in[-\pi,\pi]
\end{equation}
with periodic boundary conditions and exact solution $u(x,y,t)=\sin{(x+y-2t)}$. The expected high-order convergence is shown in Table \ref{2Dtransport_constant_table} when using Strang splitting. 
 Just like equation \eqref{1Dtransport_constant}, there is no temporal error for the convective part since the characteristics are traced exactly. 
\end{example}

\begin{table}[h!] 
\begin{center}
\caption{Convergence study with spatial mesh refinement for equation \eqref{2Dtransport_constant} with forward Euler and $CFL=8$ at $T_f=1$.}
\label{2Dtransport_constant_table}
\begin{tabular}{|*{7}{c|}} 
\hline 
\multicolumn{7}{|c|}{Strang splitting}\\
\hline
$N_x=N_y$&$L^1$ Error&Order&$L^2$ Error&Order&$L^{\infty}$ Error&Order\\
\hline
100&4.24E-09&-&7.49E-10&-&1.69E-10&-\\
200&1.27E-10&5.05&2.25E-11&5.05&5.15E-12&5.03\\
300&1.60E-11&5.11&2.84E-12&5.11&8.23E-13&4.52\\
400&3.57E-12&5.21&6.88E-13&4.92&3.68E-13&2.80\\
\hline
\end{tabular} 
\end{center} 
\end{table}

\begin{example}(Rigid body rotation)
\begin{equation}
    \label{2Dtransport_rigidbody}
    u_t - yu_x + xu_y = 0,\qquad x,y\in[-\pi,\pi]
\end{equation}
with periodic boundary conditions. We choose the exact solution $u(x,y,t)=u(x,y,t=0)=\text{exp}(-3(x^2+y^2))$ for convergence tests. The convergence results are presented in Tables \ref{2Dtransport_rigidbody_table1} and \ref{2Dtransport_rigidbody_table2}. Strang splitting dominates the error and we observe the expected second-order convergence. Whereas, the spatial error dominates when using fourth-order splitting as evidenced by the fifth-order convergence. The error plot using a fixed mesh $N_x=N_y=200$ and varying the CFL number from 0.1 to 50 is shown in Figure \ref{rigidplot1}. Second-order convergence in time is observed when using Strang splitting, and fourth-order convergence is observed when using fourth-order splitting. We note that comparable convergence results were observed for the non-symmetric initial condition $u(x,y,t=0)=\text{exp}(-3x^2-2y^2)$. To demonstrate the effectiveness of WENO-AO controlling spurious oscillations we choose the initial condition $u(x,y,t=0)=1$ if $x,y\in[-\pi/2,\pi/2]$; $u(x,y,t=0)=0$ otherwise. With a fixed mesh $N_x=N_y=100$ and $CFL=2.2$, we compute the solution up to time $T_f=2\pi$ using SSP RK3. The discontinuities are smoothed out and no spurious oscillations occur.
\end{example}

\begin{table}[h!] 
\begin{center}
\caption{Convergence study with spatial mesh refinement for equation \eqref{2Dtransport_rigidbody} with RK4 and $CFL=0.95$ at $T_f=0.5$.}
\label{2Dtransport_rigidbody_table1}
\begin{tabular}{|*{7}{c|}} 
\hline 
\multicolumn{7}{|c|}{Strang splitting}\\
\hline
$N_x=N_y$&$L^1$ Error&Order&$L^2$ Error&Order&$L^{\infty}$ Error&Order\\
\hline
100&3.57E-05&-&1.54E-05&-&2.43E-05&-\\
200&1.83E-06&4.29&9.94E-07&3.95&1.23E-06&4.31\\
300&8.01E-07&2.04&4.35E-07&2.04&4.72E-07&2.36\\
400&4.50E-07&2.01&2.44E-07&2.01&2.55E-07&2.13\\
\hline
\multicolumn{7}{|c|}{Fourth-order splitting}\\
\hline
$N_x=N_y$&$L^1$ Error&Order&$L^2$ Error&Order&$L^{\infty}$ Error&Order\\
\hline
100&1.05E-04&-&4.66E-05&-&8.95E-05&-\\
200&1.08E-06&6.60&6.13E-07&6.25&7.74E-07&6.85\\
300&1.39E-07&5.06&8.12E-08&4.99&1.02E-07&5.00\\
400&3.32E-08&4.99&1.93E-08&4.99&2.43E-08&4.99\\
\hline
\end{tabular} 
\end{center} 
\end{table} 

\begin{table}[h!] 
\begin{center}
\caption{Convergence study with spatial mesh refinement for equation \eqref{2Dtransport_rigidbody} with RK4 and $CFL=8$ at $T_f=0.5$.}
\label{2Dtransport_rigidbody_table2}
\begin{tabular}{|*{7}{c|}} 
\hline 
\multicolumn{7}{|c|}{Strang splitting}\\
\hline
$N_x=N_y$&$L^1$ Error&Order&$L^2$ Error&Order&$L^{\infty}$ Error&Order\\
\hline
100&4.97E-04&-&2.68E-04&-&2.78E-04&-\\
200&1.24E-04&2.00&6.74E-05&1.99&6.86E-05&2.02\\
300&5.59E-05&1.97&3.03E-05&1.97&3.08E-05&1.97\\
400&3.20E-05&1.94&1.73E-05&1.94&1.76E-05&1.94\\
\hline
\multicolumn{7}{|c|}{Fourth-order splitting}\\
\hline
$N_x=N_y$&$L^1$ Error&Order&$L^2$ Error&Order&$L^{\infty}$ Error&Order\\
\hline
100&4.95E-05&-&2.46E-05&-&6.82E-05&-\\
200&4.84E-07&6.68&2.92E-07&6.39&4.77E-07&7.16\\
300&6.17E-08&5.08&3.86E-08&4.99&6.19E-08&5.03\\
400&1.47E-08&4.99&9.22E-09&4.88&1.47E-08&5.00\\
\hline
\end{tabular} 
\end{center} 
\end{table} 

\begin{figure}[h!]
\begin{minipage}{0.48\linewidth}
    \centering
    \includegraphics[width=\textwidth]{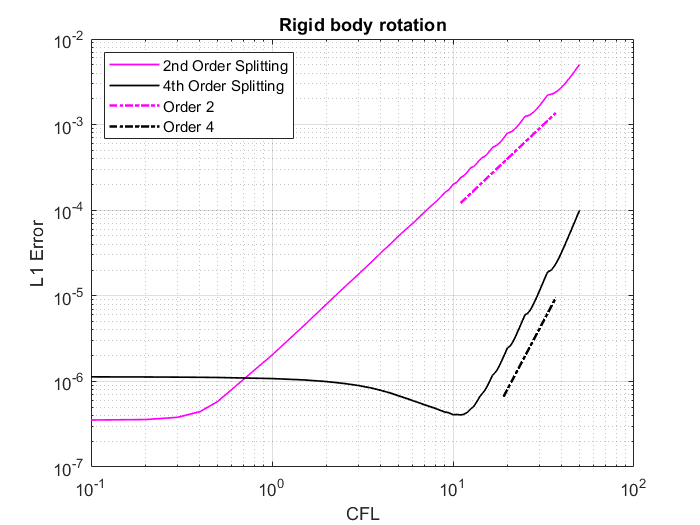}
    \caption{Error plot for \eqref{2Dtransport_rigidbody} with RK4 at $T_f=0.5$. $N_x=N_y=200$.}
    \label{rigidplot1}
\end{minipage}
\begin{minipage}{0.03\linewidth}
    \centering
    \ 
\end{minipage}
\begin{minipage}{0.48\linewidth}
    \centering
    \includegraphics[width=\textwidth]{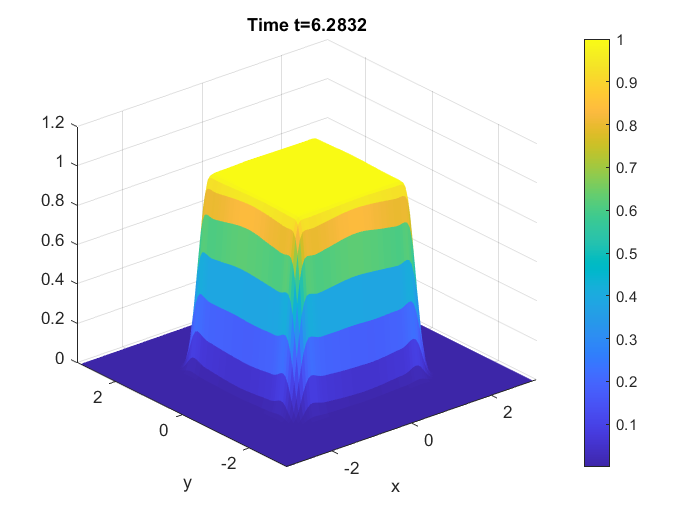}
    \caption{Plot of the numerical solution to \eqref{2Dtransport_rigidbody} with SSP RK3 and $CFL=2.2$ at $T_f=2\pi$. $N_x=N_y=100$.}
    \label{rigidplot2}
\end{minipage}
\end{figure}

\begin{example}(Swirling deformation)
\begin{equation}
    \label{2Dtransport_swirling}
    u_t - \left(\cos^2{(x/2)}\sin{(y)}g(t)u\right)_x + \left(\sin{(x)}\cos^2{(y/2)}g(t)u\right)_y = 0,\qquad x,y\in[-\pi,\pi]
\end{equation}
When testing convergence we set $g(t)=\cos{(\pi t/T_f)}\pi$ and choose the initial condition to be the smooth (with $C^5$ smoothness) cosine bell
\begin{equation}
    \label{cosbell}
    u(x,y,t=0)=
\begin{cases}
    r_0^b\cos^6{\left(\frac{r^b(x,y)\pi}{2r_0^b}\right)},&\text{if\ }r^b(x,y)<r_0^b,\\
    0,&\text{otherwise,}
\end{cases}
\end{equation}
where $r_0^b=0.3\pi$ and $r^b(x,y)=\sqrt{(x-x_0^b)^2+(y-y_0^b)^2}$ with $(x_0^b,y_0^b)=(0.3\pi,0)$. The convergence results under spatial mesh refinement are presented in Tables \ref{2Dtransport_swirling_table1} and \ref{2Dtransport_swirling_table2}. Surprisingly, high-order convergence is observed in all test cases, even for the large CFL number of 8. In particular, we observed super-convergence for $CFL=8$ when using Strang splitting. We got comparable convergence results when letting the initial condition be: (1) a cosine bell of $C^3$ smoothness, and (2) the cosine bell \eqref{cosbell} but with $x_0^b=0.6\pi$ and the width in the $x-$direction scaled by a factor of 1/2.\\
\ \\
The temporal orders of convergence are shown in Figure \ref{swirlplot1} using a fixed mesh $N_x=N_y=200$ and varying the CFL number from 0.1 to 25. When using Strang splitting the temporal convergence switches from second-order to third-order, indicating that for very large CFL numbers the splitting error does not dominate the time-stepping error as much. Fourth-order convergence is observed when using fourth-order splitting. To demonstrate the effectiveness of WENO-AO in controlling spurious oscillations we set $g(t)=1$ and choose the initial condition \cite{Leveque1996}
\begin{equation}
u(x,y,t=0) = 
\begin{cases}
    1,&\text{if\ }r^b(x,y)<r_0^b,\\
    0,&\text{otherwise,}
\end{cases}
\end{equation}
where $r_0^b=8\pi/5$ and $r^b(x,y)=\sqrt{(x-x_0^b)^2+(y-y_0^b)^2}$ with $(x_0^b,y_0^b)=(\pi,\pi)$. With a fixed mesh $N_x=N_y=100$ and $CFL=8$, we compute the solution up to time $T_f=5\pi$ using SSP RK3 and Strang splitting. The discontinuities are smoothed out and no spurious oscillations occur.
\end{example}

\begin{table}[h!] 
\begin{center}
\caption{Convergence study with spatial mesh refinement for equation \eqref{2Dtransport_swirling} with RK4 and $CFL=0.95$ at $T_f=1.5$.}
\label{2Dtransport_swirling_table1}
\begin{tabular}{|*{7}{c|}} 
\hline 
\multicolumn{7}{|c|}{Strang splitting}\\
\hline
$N_x=N_y$&$L^1$ Error&Order&$L^2$ Error&Order&$L^{\infty}$ Error&Order\\
\hline
100&5.17E-03&-&6.11E-03&-&2.04E-02&-\\
200&1.69E-04&4.94&1.69E-04&5.18&4.80E-04&5.41\\
300&3.12E-05&4.16&3.85E-05&3.64&1.41E-04&3.01\\
400&8.29E-06&4.61&1.12E-05&4.66&4.66E-05&3.86\\
\hline
\multicolumn{7}{|c|}{Fourth-order splitting}\\
\hline
$N_x=N_y$&$L^1$ Error&Order&$L^2$ Error&Order&$L^{\infty}$ Error&Order\\
\hline
100&1.42E-02&-&1.73E-02&-&5.25E-02&-\\
200&3.98E-04&5.16&3.70E-04&5.54&9.15E-04&5.84\\
300&7.69E-05&4.06&9.02E-05&3.48&3.24E-04&2.56\\
400&2.20E-05&4.35&2.89E-05&3.96&1.17E-04&3.52\\
\hline
\end{tabular} 
\end{center} 
\end{table} 

\begin{table}[h!] 
\begin{center}
\caption{Convergence study with spatial mesh refinement for equation \eqref{2Dtransport_swirling} with RK4 and $CFL=8$ at $T_f=1.5$.}
\label{2Dtransport_swirling_table2}
\begin{tabular}{|*{7}{c|}} 
\hline 
\multicolumn{7}{|c|}{Strang splitting}\\
\hline
$N_x=N_y$&$L^1$ Error&Order&$L^2$ Error&Order&$L^{\infty}$ Error&Order\\
\hline
100&1.90E-03&-&2.11E-03&-&6.83E-03&-\\
200&1.02E-04&4.23&8.88E-05&4.57&2.47E-04&4.79\\
300&1.90E-05&4.13&1.79E-05&3.94&6.36E-05&3.35\\
400&2.82E-06&6.63&4.11E-06&5.12&1.98E-05&4.05\\
\hline
\multicolumn{7}{|c|}{Fourth-order splitting}\\
\hline
$N_x=N_y$&$L^1$ Error&Order&$L^2$ Error&Order&$L^{\infty}$ Error&Order\\
\hline
100&3.89E-03&-&4.31E-03&-&1.43E-02&-\\
200&1.30E-04&4.90&1.34E-04&5.01&4.42E-04&5.01\\
300&2.41E-05&4.16&3.32E-05&3.44&1.40E-04&2.84\\
400&6.54E-06&4.53&1.02E-05&4.10&4.81E-05&3.70\\
\hline
\end{tabular} 
\end{center} 
\end{table} 

\begin{figure}[h!]
\begin{minipage}{0.48\linewidth}
    \centering
    \includegraphics[width=\textwidth]{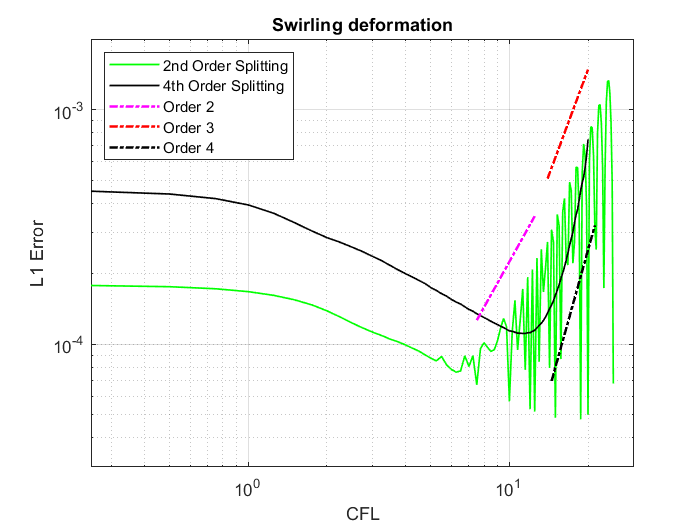}
    \caption{Error plot for \eqref{2Dtransport_swirling} with RK4 at $T_f=1.5$. $N_x=N_y=200$.}
    \label{swirlplot1}
\end{minipage}
\begin{minipage}{0.03\linewidth}
    \centering
    \ 
\end{minipage}
\begin{minipage}{0.48\linewidth}
    \centering
    \includegraphics[width=\textwidth]{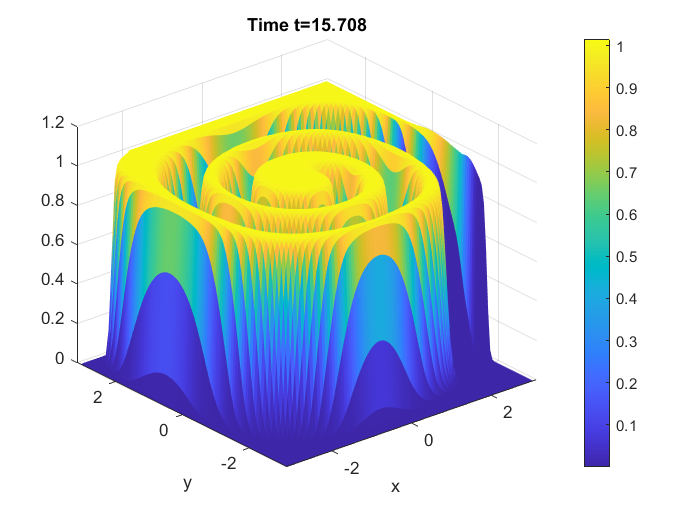}
    \caption{Plot of the numerical solution to \eqref{2Dtransport_swirling} with $g(t)=1$, SSP RK3 and $CFL=8$ at $T_f=5\pi$. $N_x=N_y=100$.}
    \label{swirlplot2}
\end{minipage}
\end{figure}


\subsection{Convection-diffusion equations: one-dimensional tests}

\begin{example}(1D equation with constant coefficient)
\begin{equation}
    \label{1D_CD_constant}
    u_t + u_x = \epsilon u_{xx},\qquad x\in[0,2\pi]
\end{equation}
with periodic boundary conditions and exact solution $u(x,t)=\sin{(x-t)}\text{exp}(-\epsilon t)$. We set $\epsilon=1$. The convergence results under spatial mesh refinement are shown in Table \ref{1D_CD_constant_table} for $CFL=0.95$ and $CFL=8$. In both cases we observe the expected third-order convergence since we are using IMEX(2,3,3) for the time-stepping. Note that there is no temporal error for the convective part since the characteristics are traced exactly and hence $F_{j+\frac{1}{2}}(t)=0$ for all $t\in[t^n,t^{n+1}]$ and $j=1,2,...,N_x$. Figure \ref{1D_CD_constant_plot} shows the expected third-order convergence in time using fixed mesh $N_x=400$ and varying the CFL number from 0.1 to 15.
\end{example}

\begin{table}[h!] 
\begin{center}
\caption{Convergence study with spatial mesh refinement for equation \eqref{1D_CD_constant} with IMEX(2,3,3) at $T_f=1$.}
\label{1D_CD_constant_table}
\begin{tabular}{|*{7}{c|}} 
\hline 
\multicolumn{7}{|c|}{$CFL=0.95$}\\
\hline
$N_x$&$L^1$ Error&Order&$L^2$ Error&Order&$L^{\infty}$ Error&Order\\
\hline
50&1.87E-04&-&8.26E-05&-&4.66E-05&-\\
100&2.55E-05&2.87&1.13E-05&2.87&6.36E-06&2.87\\
200&3.34E-06&2.93&1.48E-06&2.93&8.35E-07&2.93\\
400&4.31E-07&2.95&1.91E-07&2.95&1.08E-07&2.95\\
800&5.44E-08&2.99&2.41E-08&2.99&1.36E-08&2.99\\
\hline
\multicolumn{7}{|c|}{$CFL=8$}\\
\hline
$N_x$&$L^1$ Error&Order&$L^2$ Error&Order&$L^{\infty}$ Error&Order\\
\hline
50&6.87E-02&-&3.04E-02&-&1.72E-02&-\\
100&1.09E-02&2.66&4.82E-03&2.66&2.72E-03&2.66\\
200&1.63E-03&2.74&7.22E-04&2.74&4.07E-04&2.74\\
400&2.27E-04&2.84&1.01E-04&2.84&5.69E-05&2.84\\
800&3.03E-05&2.91&1.34E-05&2.91&7.57E-06&2.91\\
\hline
\end{tabular} 
\end{center} 
\end{table} 

\begin{example}(1D equation with variable coefficient)
\begin{equation}
    \label{1D_CD_variable}
    u_t + (\sin{(x)}u)_x = \epsilon u_{xx} + g,\qquad x\in[0,2\pi]
\end{equation}
with periodic boundary conditions and $g(x,t)=\sin{(2x)}\text{exp}(-\epsilon t)$ and exact solution $u(x,t)=\sin{(x)}\text{exp}(-\epsilon t)$. We set $\epsilon=1$. Table \ref{1D_CD_variable_table} shows the convergence results under spatial mesh refinement. Third-order convergence in space is observed for $CFL=0.95$. Whereas, the convergence for $CFL=8$ is roughly order 2.6 since the time-stepping start to dominate. We note that the order of convergence for IMEX(2,3,3) under increasing the CFL number dips slightly below three for larger CFL numbers. We use fixed mesh $N_x=400$ and vary the CFL number from 0.1 to 15 for the error plot showing the temporal order of convergence in Figure \ref{1D_CD_variable_plot}.
\end{example}

\begin{table}[h!] 
\begin{center}
\caption{Convergence study with spatial mesh refinement for equation \eqref{1D_CD_variable} with IMEX(2,3,3) at $T_f=1$.}
\label{1D_CD_variable_table}
\begin{tabular}{|*{7}{c|}} 
\hline 
\multicolumn{7}{|c|}{$CFL=0.95$}\\
\hline
$N_x$&$L^1$ Error&Order&$L^2$ Error&Order&$L^{\infty}$ Error&Order\\
\hline
50&2.10E-03&-&9.99E-04&-&7.70E-04&-\\
100&3.99E-04&2.39&1.88E-04&2.41&1.44E-04&2.42\\
200&6.41E-05&2.64&2.99E-05&2.65&2.27E-05&2.66\\
400&9.84E-06&2.70&4.57E-06&2.71&3.44E-06&2.73\\
800&1.30E-06&2.92&6.04E-07&2.92&4.53E-07&2.92\\
\hline
\multicolumn{7}{|c|}{$CFL=8$}\\
\hline
$N_x$&$L^1$ Error&Order&$L^2$ Error&Order&$L^{\infty}$ Error&Order\\
\hline
50&7.50E-01&-&4.34E-01&-&3.92E-01&-\\
100&1.26E-01&2.58&6.60E-02&2.72&6.63E-02&2.56\\
200&1.74E-02&2.85&8.31E-03&2.99&6.50E-03&3.35\\
400&3.09E-03&2.50&1.46E-03&2.51&1.12E-03&2.53\\
800&5.03E-04&2.62&2.36E-04&2.63&1.81E-04&2.63\\
\hline
\end{tabular} 
\end{center} 
\end{table} 

\begin{figure}[h!]
\begin{minipage}{0.48\linewidth}
    \centering
    \includegraphics[width=\textwidth]{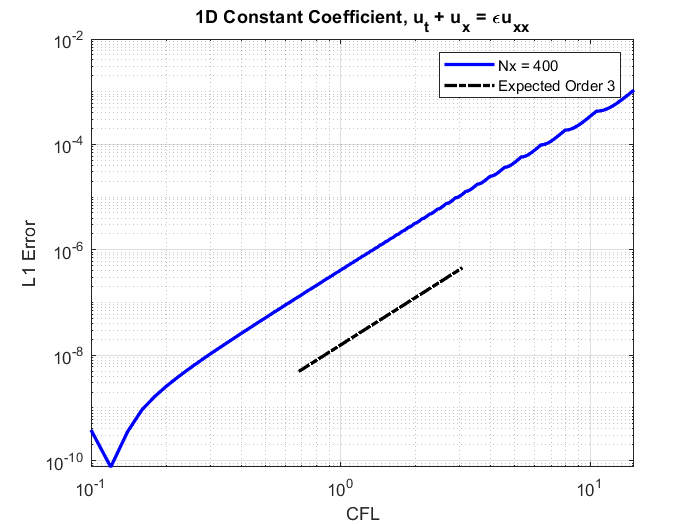}
    \caption{IMEX(2,3,3), $\epsilon = 1$, Final time $T_f=0.5$.}
    \label{1D_CD_constant_plot}
\end{minipage}
\begin{minipage}{0.03\linewidth}
    \centering
    \ 
\end{minipage}
\begin{minipage}{0.48\linewidth}
    \centering
    \includegraphics[width=\textwidth]{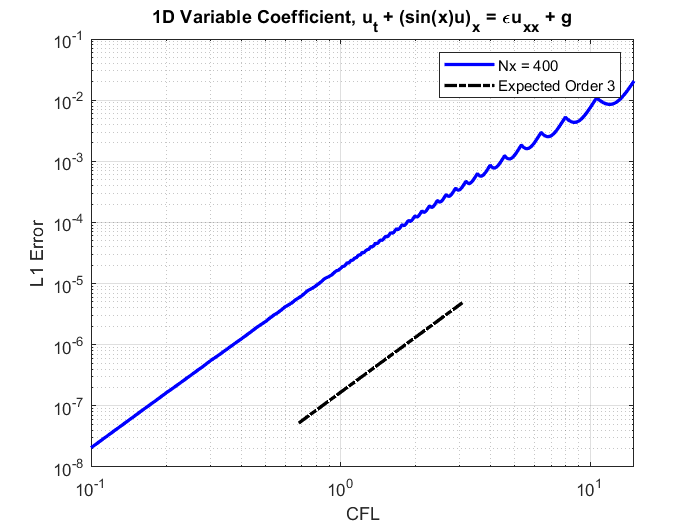}
    \caption{IMEX(2,3,3), $\epsilon = 1$, Final time $T_f=0.5$.}
    \label{1D_CD_variable_plot}
\end{minipage}
\end{figure}

\begin{example}(1D viscous Burgers' equation)
\begin{equation}
    \label{1D_viscous_burgers}
    u_t + \left(\frac{u^2}{2}\right) = \epsilon u_{xx},\qquad x\in[0,2]
\end{equation}
with periodic boundary conditions. As in \cite{Seydaoglu2016}, we set $\epsilon=0.1$ and choose the initial condition $u(x,t=0)=0.2\sin{(\pi x)}$. The exact solution is
\[u(x,t) = 2\epsilon\pi\frac{\sum\limits_{n=1}^{\infty}{c_n\text{exp}(-n^2\pi^2\epsilon t)n\sin{(n\pi x)}}}{c_0 + \sum\limits_{n=1}^{\infty}{c_n\text{exp}(-n^2\pi^2\epsilon t)\cos{(n\pi x)}}},\]
where $c_0 = \int_0^1{\text{exp}(-(1-\cos{(\pi x)})/(10\pi\epsilon))dx}$ and $c_n=2\int_0^1{\text{exp}(-(1-\cos{(\pi x)})/(10\pi\epsilon))\cos{(n\pi x)}dx}$ for $n=1,2,3,...$. We computed the first ten Fourier coefficients in Mathematica$^{\text{\textregistered}}$ for the exact solution; the eleventh Fourier coefficient was less than machine precision.\\
\ \\
Table \ref{1D_viscous_burgers_table} shows the expected orders of convergence under spatial mesh refinement. Third-order convergence is observed for $CFL=0.95$. Whereas, the convergence for $CFL=8$ is slightly below order three since the time-stepping error starts to dominate. W note that the order of convergence for IMEX(2,3,3) under increasing the CFL number dips slightly below three for larger CFL numbers. The error plot in Figure \ref{1D_viscous_burgers_plot} showing third-order convergence in time uses mesh $N_x=400$ and CFL numbers varying from 0.1 to 15.
\end{example}

\begin{table}[h!] 
\begin{center}
\caption{Convergence study with spatial mesh refinement for equation \eqref{1D_viscous_burgers} with IMEX(2,3,3) at $T_f=1$.}
\label{1D_viscous_burgers_table}
\begin{tabular}{|*{7}{c|}} 
\hline 
\multicolumn{7}{|c|}{$CFL=0.95$}\\
\hline
$N_x$&$L^1$ Error&Order&$L^2$ Error&Order&$L^{\infty}$ Error&Order\\
\hline
50&9.50E-05&-&8.43E-05&-&1.17E-04&-\\
100&1.48E-05&2.68&1.32E-05&3.67&1.87E-05&2.65\\
200&2.42E-06&2.62&2.20E-06&2.59&3.21E-06&2.55\\
400&3.27E-07&2.88&3.00E-07&2.87&4.42E-07&2.86\\
800&4.28E-08&2.94&3.95E-08&2.93&5.86E-08&2.92\\
\hline
\multicolumn{7}{|c|}{$CFL=8$}\\
\hline
$N_x$&$L^1$ Error&Order&$L^2$ Error&Order&$L^{\infty}$ Error&Order\\
\hline
50&1.26E-02&-&1.31E-02&-&2.24E-02&-\\
100&2.79E-03&2.17&2.50E-03&2.38&3.70E-03&2.60\\
200&4.73E-04&2.56&4.16E-04&2.59&5.64E-04&2.71\\
400&1.28E-04&1.89&1.15E-04&1.85&1.64E-04&1.78\\
800&1.97E-05&2.70&1.78E-05&2.70&2.56E-05&2.68\\
\hline
\end{tabular} 
\end{center} 
\end{table}

\begin{example}(The 0D1V Leonard-Bernstein (linearized) Fokker-Planck equation)
\begin{equation}
    \label{1DFP}
    f_t -\frac{1}{\epsilon}((v_x - \overline{v}_x)f)_{v_x}= \frac{1}{\epsilon}Df_{v_xv_x},\qquad v_x\in[-2\pi,2\pi]
\end{equation}
with zero boundary conditions and equilibrium solution the Maxwellian
\begin{equation}
    f_M(v_x) = \frac{n}{\sqrt{2\pi RT}}\text{exp}\left(-\frac{(v_x - \overline{v}_x)^2}{2RT}\right),
\end{equation}
where $\epsilon=1$, gas constant $R=1/6$, temperature $T=3$, thermal velocity $v_{th}=\sqrt{2RT}=\sqrt{2D}=1$, number density $n=\pi$, and bulk velocity $\overline{v}_x=0$. These quantities were chosen for scaling convenience. When testing convergence we set the initial distribution $f(v_x,t=0)=f_M(v_x)$. Table \ref{1DFP_table} shows the convergence results, for which we use IMEX(4,4,3) for the time-stepping; we show the results using IMEX(4,4,3) because it gave slightly better convergence than IMEX(2,3,3). We observe fourth-order convergence under spatial mesh refinement for $CFL=0.95$. Whereas, for $CFL=8$ the time-stepping error starts to dominate and we observe third-order convergence. The error plot in Figure \ref{1DFP_plot} showing third-order convergence in time uses a fixed mesh $N_{v_x}=400$ and CFL numbers varying from 0.1 to 15. We note that although high-order convergence is observed, the proposed EL-RK-FV algorithm is not equilibrium-preserving.
\end{example}

\begin{table}[h!] 
\begin{center}
\caption{Convergence study with spatial mesh refinement for equation \eqref{1DFP} with IMEX(4,4,3) at $T_f=1$.}
\label{1DFP_table}
\begin{tabular}{|*{7}{c|}} 
\hline 
\multicolumn{7}{|c|}{$CFL=0.95$}\\
\hline
$N_{v_x}$&$L^1$ Error&Order&$L^2$ Error&Order&$L^{\infty}$ Error&Order\\
\hline
50&8.02E-04&-&5.19E-04&-&5.65E-04&-\\
100&6.12E-05&3.71&3.83E-05&3.76&4.27E-05&3.73\\
200&4.41E-06&3.79&2.63E-06&3.87&2.84E-06&3.91\\
400&2.99E-07&3.88&1.73E-07&3.93&1.80E-07&3.98\\
800&2.03E-08&3.88&1.13E-08&3.94&1.10E-08&4.04\\
\hline
\multicolumn{7}{|c|}{$CFL=8$}\\
\hline
$N_{v_x}$&$L^1$ Error&Order&$L^2$ Error&Order&$L^{\infty}$ Error&Order\\
\hline
50&9.33E-03&-&4.52E-03&-&3.63E-03&-\\
100&1.34E-03&2.80&6.60E-04&2.78&5.63E-04&2.69\\
200&1.91E-04&2.81&9.57E-05&2.79&8.15E-05&2.79\\
400&3.21E-05&2.57&1.61E-05&2.57&1.34E-05&2.61\\
800&4.13E-06&2.96&2.09E-06&2.95&1.74E-06&2.94\\
\hline
\end{tabular} 
\end{center} 
\end{table} 

\begin{figure}[h!]
\begin{minipage}{0.48\linewidth}
    \centering
    \includegraphics[width=\textwidth]{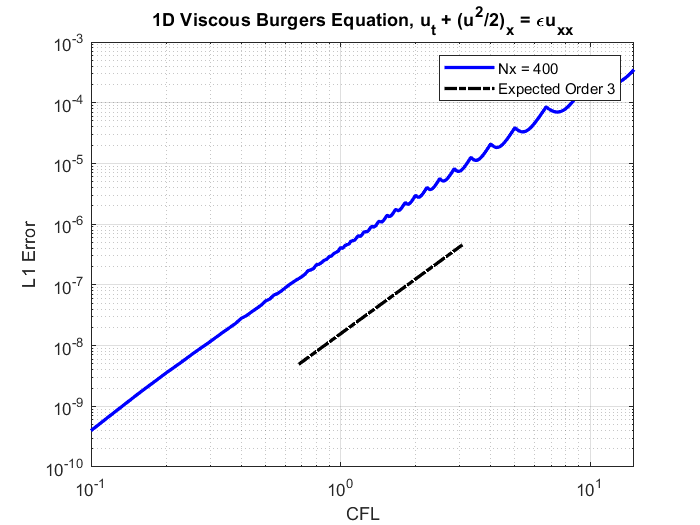}
    \caption{IMEX(2,3,3), $\epsilon = 0.1$, Final time $T_f=0.5$.}
    \label{1D_viscous_burgers_plot}
\end{minipage}
\begin{minipage}{0.03\linewidth}
    \centering
    \ 
\end{minipage}
\begin{minipage}{0.48\linewidth}
    \centering
    \includegraphics[width=\textwidth]{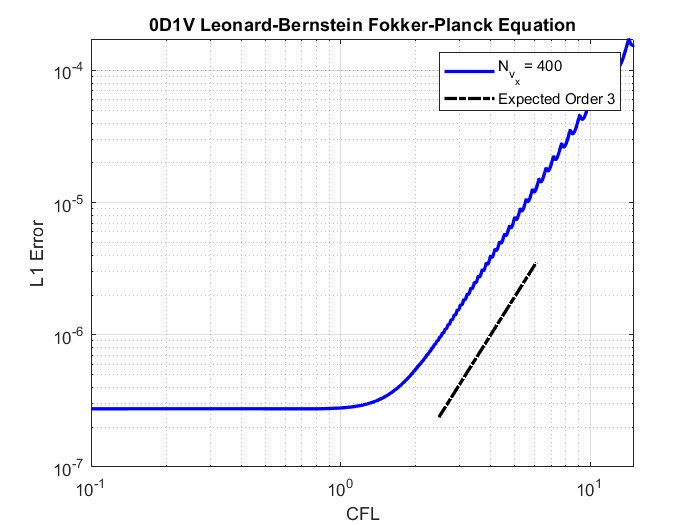}
    \caption{IMEX(4,4,3), Final time $T_f=0.5$.}
    \label{1DFP_plot}
\end{minipage}
\end{figure}

\subsection{Convection-diffusion equations: two-dimensional tests}

\begin{example}(2D equation with constant coefficient)
\begin{equation}
    \label{2D_CD_constant}
    u_t + u_x + u_y = \epsilon(u_{xx} + u_{yy}),\qquad x,y\in[0,2\pi]
\end{equation}
with periodic boundary conditions and exact solution $u(x,y,t)=\text{exp}(-2\epsilon t)\sin{(x+y-2t)}$. We set $\epsilon=1$. Third-order convergence under spatial mesh refinement is seen in Table \ref{2D_CD_constant_table} for $CFL=0.95$ and $CFL=8$. As with equation \eqref{1D_CD_constant}, there is no temporal error for the convective part since the characteristics are traced exactly. Note that the error is larger for $CFL=8$ than $CFL=0.95$ since this problem also has diffusion. Figure \ref{2D_CD_constant_plot} shows the third-order convergence in time using fixed mesh $N_x=N_y=400$ and varying the CFL number from 6 to 20.
\end{example}

\begin{table}[h!] 
\begin{center}
\caption{Convergence study with spatial mesh refinement for equation \eqref{2D_CD_constant} with IMEX(2,3,3) and Strang splitting at $T_f=0.5$.}
\label{2D_CD_constant_table}
\begin{tabular}{|*{7}{c|}} 
\hline 
\multicolumn{7}{|c|}{$CFL=0.95$}\\
\hline
$N_x=N_y$&$L^1$ Error&Order&$L^2$ Error&Order&$L^{\infty}$ Error&Order\\
\hline
50&6.74E-05&-&1.19E-05&-&2.68E-06&-\\
100&1.02E-05&2.72&1.81E-06&2.72&4.07E-07&2.72\\
200&1.41E-06&2.86&2.49E-07&2.86&5.62E-08&2.86\\
400&1.86E-07&2.92&3.29E-08&2.92&7.41E-09&2.92\\
\hline
\multicolumn{7}{|c|}{$CFL=8$}\\
\hline
$N_x=N_y$&$L^1$ Error&Order&$L^2$ Error&Order&$L^{\infty}$ Error&Order\\
\hline
50&3.92E-02&-&6.94E-03&-&1.56E-03&-\\
100&5.82E-03&2.75&1.03E-03&2.75&2.32E-04&2.75\\
200&8.09E-04&2.85&1.43E-04&2.85&3.22E-05&2.85\\
400&1.07E-04&2.92&1.90E-05&2.92&4.27E-06&2.92\\
\hline
\end{tabular} 
\end{center} 
\end{table}

\begin{example}(Rigid body rotation with diffusion)
\begin{equation}
    \label{2D_CD_rigidbody}
    u_t - yu_x + xu_y = \epsilon(u_{xx} + u_{yy}) + g,\qquad x,y\in[-2\pi,2\pi]
\end{equation}
with periodic boundary conditions, $g(x,y,t)=(6\epsilon-4xy-4\epsilon(x^2+9y^2))\text{exp}(-(x^2+3y^2+2\epsilon t))$, and exact solution $u(x,y,t)=\text{exp}(-(x^2+3y^2+2\epsilon t))$. We set $\epsilon=1$. Table \ref{2D_CD_rigidbody_table2} shows the order of convergence when using IMEX(4,4,3). We use IMEX(4,4,3) instead of IMEX(2,3,3) because the latter choice, along with the Strang splitting, showed an order of convergence less than two for large CFL numbers. The expected second-order convergence in time (due to Strang splitting) is seen in Figure \ref{2D_CD_rigidbody_plot} assumes fixed mesh $N_x=N_y=400$ and CFL numbers varying from 6 to 20.
\end{example}

\begin{table}[h!] 
\begin{center}
\caption{Convergence study with spatial mesh refinement for equation \eqref{2D_CD_rigidbody} with IMEX(4,4,3) and Strang splitting at $T_f=0.5$.}
\label{2D_CD_rigidbody_table2}
\begin{tabular}{|*{7}{c|}} 
\hline
\multicolumn{7}{|c|}{$CFL=0.95$}\\
\hline
$N_x=N_y$&$L^1$ Error&Order&$L^2$ Error&Order&$L^{\infty}$ Error&Order\\
\hline
50&3.22E-03&-&1.42E-03&-&1.44E-03&-\\
100&3.92E-04&3.04&1.75E-04&3.02&2.00E-04&2.85\\
200&7.27E-05&2.43&3.14E-05&2.47&3.54E-05&2.50\\
400&1.65E-05&2.14&7.11E-06&2.15&7.80E-06&2.18\\
\hline
\multicolumn{7}{|c|}{$CFL=5$}\\
\hline
$N_x=N_y$&$L^1$ Error&Order&$L^2$ Error&Order&$L^{\infty}$ Error&Order\\
\hline
50&2.21E-02&-&8.81E-03&-&7.44E-03&-\\
100&5.98E-03&1.89&2.48E-03&1.83&2.46E-03&1.60\\
200&1.61E-03&1.89&6.81E-04&1.86&7.14E-04&1.79\\
400&4.24E-04&1.93&1.81E-04&1.91&1.94E-04&1.88\\
\hline
\end{tabular} 
\end{center} 
\end{table} 

\begin{figure}[h!]
\begin{minipage}{0.48\linewidth}
    \centering
    \includegraphics[width=\textwidth]{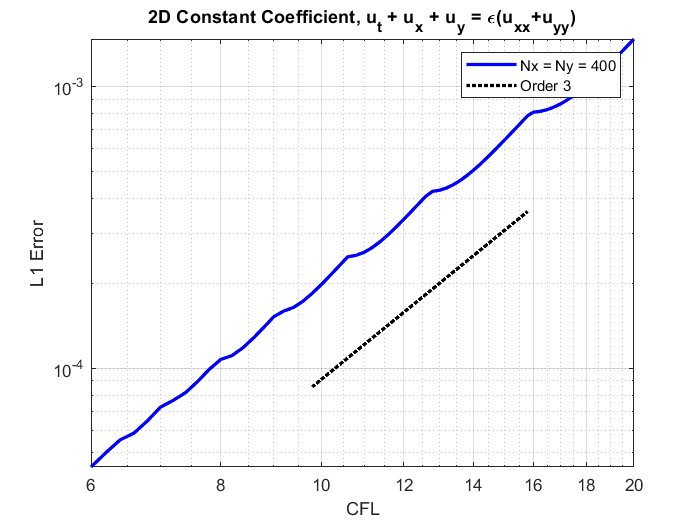}
    \caption{IMEX(2,3,3), $\epsilon = 1$, Final time $T_f=0.5$.}
    \label{2D_CD_constant_plot}
\end{minipage}
\begin{minipage}{0.03\linewidth}
    \centering
    \ 
\end{minipage}
\begin{minipage}{0.48\linewidth}
    \centering
    \includegraphics[width=\textwidth]{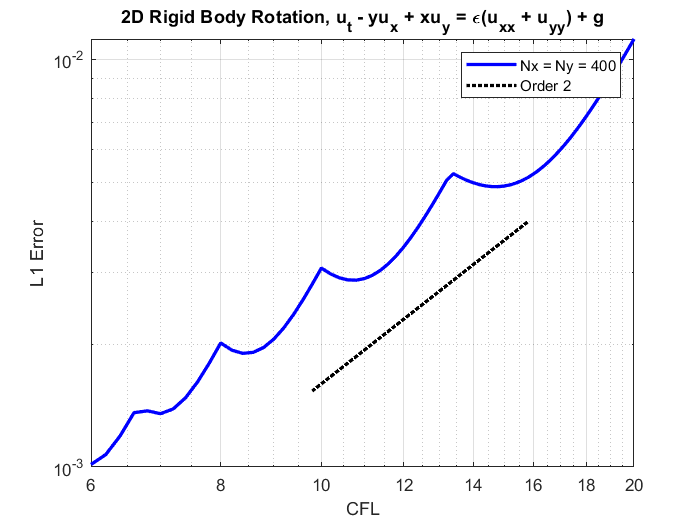}
    \caption{IMEX(4,4,3), $\epsilon = 1$, Final time $T_f=0.1$.}
    \label{2D_CD_rigidbody_plot}
\end{minipage}
\end{figure}

\begin{example}(Swirling deformation with diffusion)
\begin{equation}
    \label{2D_CD_swirling}
    u_t - \left(\cos^2{(x/2)}\sin{(y)}f(t)u\right)_x + \left(\sin{(x)}\cos^2{(y/2)}f(t)u\right)_y = \epsilon(u_{xx} + u_{yy}),\qquad x,y\in[-\pi,\pi]
\end{equation}
When testing the convergence we set $f(t)=\cos{(\pi t/T_f)}\pi$, $\epsilon=1$, and choose the initial condition to be the cosine bell in equation \eqref{cosbell}. Since there is no analytic solution, we use a reference solution computed with a mesh size of $400\times 400$ and $CFL=0.1$. The convergence results under spatial mesh refinement are presented in Table \ref{2D_CD_swirling_table2}. The splitting error seems to dominate the time-stepping error for $CFL=0.95$ as evidenced by the apparent second-order convergence. Whereas, the time-stepping error seems to contribute more for $CFL=8$. Due to the interplay between the time-stepping and splitting errors, the temporal order 2.4 is also observed in Figure \ref{2D_CD_swirling_plot}, for which we use fixed mesh $N_x=N_y=400$ and CFL numbers varying from 6 to 20.
\end{example}

\begin{table}[h!] 
\begin{center}
\caption{Convergence study with spatial mesh refinement for equation \eqref{2D_CD_swirling} with IMEX(2,3,3) and Strang splitting at $T_f=0.1$.}
\label{2D_CD_swirling_table2}
\begin{tabular}{|*{7}{c|}} 
\hline 
\multicolumn{7}{|c|}{$CFL=0.95$}\\
\hline
$N_x=N_y$&$L^1$ Error&Order&$L^2$ Error&Order&$L^{\infty}$ Error&Order\\
\hline
50&4.98E-04&-&2.57E-04&-&3.38E-04&-\\
100&9.13E-05&2.45&4.58E-05&2.49&6.05E-05&2.48\\
200&1.91E-05&2.26&9.56E-06&2.26&1.22E-05&2.31\\
400&4.48E-06&2.09&2.19E-06&2.13&2.61E-06&2.23\\
\hline
\multicolumn{7}{|c|}{$CFL=8$}\\
\hline
$N_x=N_y$&$L^1$ Error&Order&$L^2$ Error&Order&$L^{\infty}$ Error&Order\\
\hline
50&5.93E-02&-&3.60E-02&-&5.82E-02&-\\
100&2.05E-02&1.53&1.16E-02&1.63&1.65E-02&1.82\\
200&3.14E-03&2.71&1.67E-03&2.80&2.04E-03&3.01\\
400&6.08E-04&2.37&3.13E-04&2.41&4.16E-04&2.30\\
\hline
\end{tabular} 
\end{center} 
\end{table}

\begin{example}(2D viscous Burgers' equation)
\begin{equation}
    \label{2D_viscousburgers}
    u_t + \left(\frac{u^2}{2}\right)_x + \left(\frac{u^2}{2}\right)_y = \epsilon(u_{xx} + u_{yy}) + g,\qquad x,y\in[-\pi,\pi]
\end{equation}
with periodic boundary conditions. As in \cite{Wang2016}, we set $\epsilon=0.1$, $g(x,y,t)=\text{exp}(-4\epsilon t)\sin{(2(x+y))}$, and suppose the solution $u(x,y,t)=\text{exp}(-2\epsilon t)\sin{(x+y)}$. The convergence results are presented in Table \ref{2D_viscousburgers_table2}. The splitting error seems to dominate the time-stepping error for $CFL=0.95$ as evidenced by the second-order convergence. Whereas, the time-stepping error seems to contribute more for $CFL=8$ since the order is between two and three. The temporal order of convergence in the $L^1$ norm is roughly 2.3, as seen in Figure \ref{2D_viscousburgers_plot}, for which we use fixed mesh $N_x=N_y=400$ and CFL numbers varying from 6 to 20.
\end{example}

\begin{table}[h!] 
\begin{center}
\caption{Convergence study with spatial mesh refinement for equation \eqref{2D_viscousburgers} with IMEX(2,3,3) and Strang splitting at $T_f=0.5$.}
\label{2D_viscousburgers_table2}
\begin{tabular}{|*{7}{c|}} 
\hline 
\multicolumn{7}{|c|}{$CFL=0.95$}\\
\hline
$N_x=N_y$&$L^1$ Error&Order&$L^2$ Error&Order&$L^{\infty}$ Error&Order\\
\hline
50&3.01E-04&-&5.21E-05&-&1.19E-05&-\\
100&7.12E-05&2.08&1.30E-05&2.00&3.42E-06&1.80\\
200&1.76E-05&2.02&3.34E-06&1.96&9.42E-07&1.86\\
400&4.53E-06&1.95&8.60E-07&1.96&2.49E-07&1.92\\
\hline
\multicolumn{7}{|c|}{$CFL=8$}\\
\hline
$N_x=N_y$&$L^1$ Error&Order&$L^2$ Error&Order&$L^{\infty}$ Error&Order\\
\hline
50&8.11E-02&-&1.60E-02&-&5.80E-03&-\\
100&1.02E-02&2.99&2.00E-03&3.00&7.78E-04&2.90\\
200&1.61E-03&2.66&3.01E-04&2.74&9.81E-05&2.99\\
400&3.47E-04&2.21&5.97E-05&2.33&1.35E-05&2.86\\
\hline
\end{tabular} 
\end{center} 
\end{table} 

\begin{figure}[h!]
\begin{minipage}{0.48\linewidth}
    \centering
    \includegraphics[width=\textwidth]{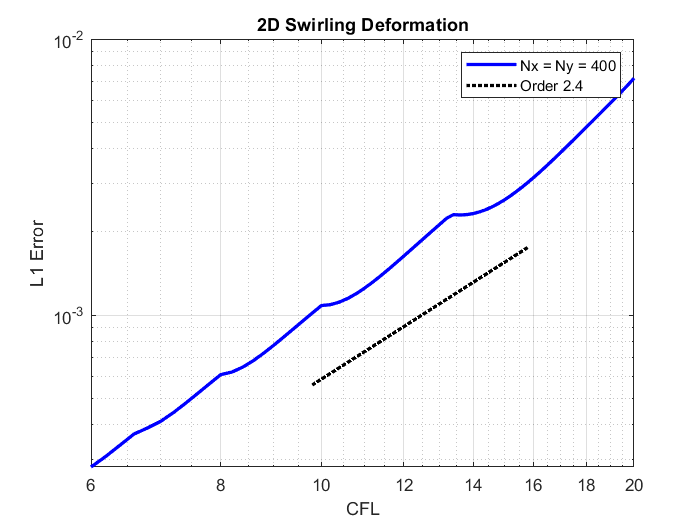}
    \caption{IMEX(2,3,3), $\epsilon = 1$, Final time $T_f=0.1$.}
    \label{2D_CD_swirling_plot}
\end{minipage}
\begin{minipage}{0.03\linewidth}
    \centering
    \ 
\end{minipage}
\begin{minipage}{0.48\linewidth}
    \centering
    \includegraphics[width=\textwidth]{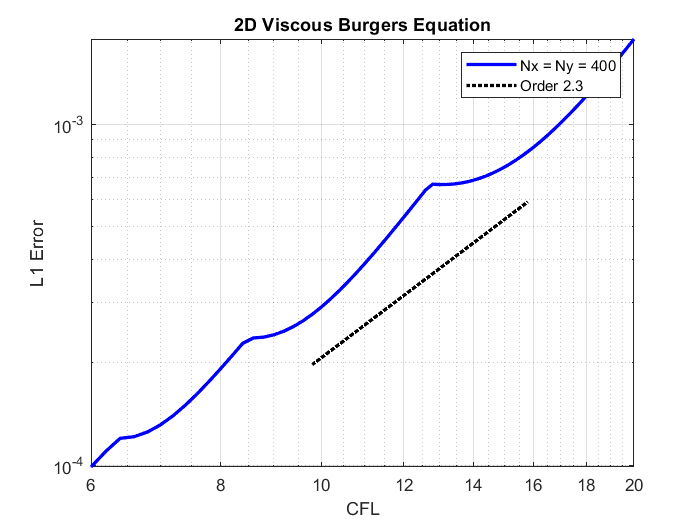}
    \caption{IMEX(2,3,3), $\epsilon = 0.1$, Final time $T_f=0.2$.}
    \label{2D_viscousburgers_plot}
\end{minipage}
\end{figure}

\begin{example}(The 0D2V Leonard-Bernstein (linearized) Fokker-Planck equation)
\begin{equation}
    \label{2DFP}
    f_t -\frac{1}{\epsilon}((v_x - \overline{v}_x)f)_{v_x} - \frac{1}{\epsilon}((v_y - \overline{v}_y)f)_{v_y}= \frac{1}{\epsilon}D(f_{v_xv_x} + f_{v_yv_y}),\qquad v_x,v_y\in[-2\pi,2\pi]
\end{equation}
with zero boundary conditions and equilibrium solution the Maxwellian
\begin{equation}
    f_M(v_x,v_y) = \frac{n}{2\pi RT}\text{exp}\left(-\frac{(v_x - \overline{v}_x)^2 + (v_y-\overline{v}_y)^2}{2RT}\right),
\end{equation}
where $\epsilon=1$, gas constant $R=1/6$, temperature $T=3$, thermal velocity $v_{th}=\sqrt{2RT}=\sqrt{2D}=1$, number density $n=\pi$, and bulk velocities $\overline{v}_x=\overline{v}_y=0$. These quantities were chosen for scaling convenience. When testing the spatial and temporal orders of accuracy we set the initial distribution $f(v_x,v_y,t=0)=f_M(v_x,v_y)$. Table \ref{2DFP_table} shows the convergence results under spatial mesh refinement. We observe higher fourth-order convergence in space for $CFL=0.95$. The time-stepping and splitting errors start to dominate the spatial error for larger CFL numbers, as observed for $CFL=8$. Figure \ref{2DFP_plot} shows the temporal order of convergence is roughly 2.6 for fixed mesh $N_{v_x}=N_{v_y}=400$ and CFL numbers varying from 6 to 20. We again note the interplay between the third-order time-stepping and second-order splitting.\\
\ \\
When testing for relaxation of the system, we choose the initial distribution $f(v_x,v_y,t=0)=f_{M1}(v_x,v_y)+f_{M2}(v_x,v_y)$, that is, the sum of two randomly generated Maxwellians such that the total macro-parameters are preserved. The number density, bulk velocities, and temperature of each Maxwellian are listed in Table \ref{2maxwells}. We set $\overline{v}_y=0$ so that the two generated Maxwellians are shifted only along the $v_x$ axis.\\

\begin{table}[h!] 
\begin{center}
\label{}
\begin{tabular}{|c|c|c|} 
\hline 
&$f_{M1}$&$f_{M2}$\\
\hline
$n$&1.990964530353041&1.150628123236752\\
\hline
$\overline{v}_x$&0.4979792385268875&-0.8616676237412346\\
\hline
$\overline{v}_y$&0&0\\
\hline
$T$&2.46518981703837&0.4107062104302872\\
\hline
\end{tabular} 
\caption{$n=\pi$, $\overline{\mathbf{v}}=\mathbf{0}$, and $T=3$.}
\label{2maxwells}
\end{center} 
\end{table}
\noindent The macro-parameters we want to conserve are number density, bulk velocity, and temperature, which in two dimensions are respectively given by
\begin{subequations}
\begin{equation}
    n = \int_{-\infty}^{\infty}{\int_{-\infty}^{\infty}{f(\mathbf{v})dv_ydv_x}},
\end{equation}
\begin{equation}
    \overline{\mathbf{v}} = \frac{1}{n}\int_{-\infty}^{\infty}{\int_{-\infty}^{\infty}{\mathbf{v}f(\mathbf{v})dv_ydv_x}},
\end{equation}
\begin{equation}
    T = \frac{1}{2Rn}\int_{-\infty}^{\infty}{\int_{-\infty}^{\infty}{(\mathbf{v}-\overline{\mathbf{v}})^2f(\mathbf{v})dv_ydv_x}}.
\end{equation}
\end{subequations}
Figures \ref{FPtwoMaxwell_images}(f) and \ref{FPtwoMaxwell_snapshots} show the solution using fixed mesh $N_{v_x}=N_{v_y}=200$ and $CFL=6$. Although we computed the solution up to time $T_f=20$, there was no difference (to the naked eye) after time $t=3$. Although the solution appears to reach equilibrium, we again note that the proposed EL-RK-FV algorithm is not equilibrium-preserving. Figure \ref{FPtwoMaxwell_images}(a) verifies mass conservation, but Figure \ref{FPtwoMaxwell_images}(b) implies that the numerical solution has some negative values and is not positivity-preserving. Referring to Figure \ref{FPtwoMaxwell_images}, momentum and energy are not conserved. As seen in Figure \ref{FPtwoMaxwell_images}(d), the bulk velocity in the $v_y$-direction is on the order of machine epsilon because we constructed the two Maxwellians in Table \ref{2maxwells} such that $\overline{v}_{M1,y}=\overline{v}_{M2,y}=0$. Hence, there is no drift velocity in the $v_y$-direction.
\end{example}

\begin{table}[h!] 
\begin{center}
\caption{Convergence study with spatial mesh refinement for equation \eqref{2DFP} with IMEX(2,3,3) and Strang splitting at $T_f=0.5$.}
\label{2DFP_table}
\begin{tabular}{|*{7}{c|}} 
\hline 
\multicolumn{7}{|c|}{$CFL=0.95$}\\
\hline
$N_{v_x}=N_{v_y}$&$L^1$ Error&Order&$L^2$ Error&Order&$L^{\infty}$ Error&Order\\
\hline
50&9.07E-04&-&4.22E-04&-&5.49E-04&-\\
100&7.19E-05&3.66&3.15E-05&3.74&4.36E-05&3.66\\
200&5.35E-06&3.75&2.15E-06&3.87&2.93E-06&3.89\\
400&3.54E-07&3.92&1.37E-07&3.98&1.82E-07&4.01\\
\hline
\multicolumn{7}{|c|}{$CFL=8$}\\
\hline
$N_{v_x}=N_{v_y}$&$L^1$ Error&Order&$L^2$ Error&Order&$L^{\infty}$ Error&Order\\
\hline
50&5.70E-03&-&1.84E-03&-&1.26E-03&-\\
100&1.08E-03&2.40&3.53E-04&2.39&2.82E-04&2.16\\
200&1.69E-04&2.67&5.68E-05&2.64&5.02E-05&2.49\\
400&2.73E-05&2.63&9.30E-06&2.61&8.63E-06&2.54\\
\hline
\end{tabular} 
\end{center} 
\end{table} 

\begin{figure}[h!]
\centering
\includegraphics[width=0.49\textwidth]{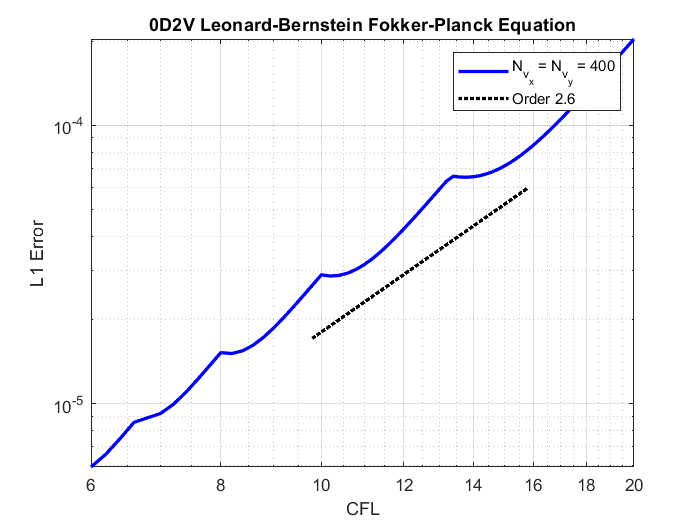}
\caption{IMEX(2,3,3), Final time $T_f=0.1$.}
\label{2DFP_plot}
\end{figure}

\begin{figure}[h!]
\begin{minipage}[b]{0.33\linewidth}
    \centering
    \includegraphics[width=1.05\textwidth]{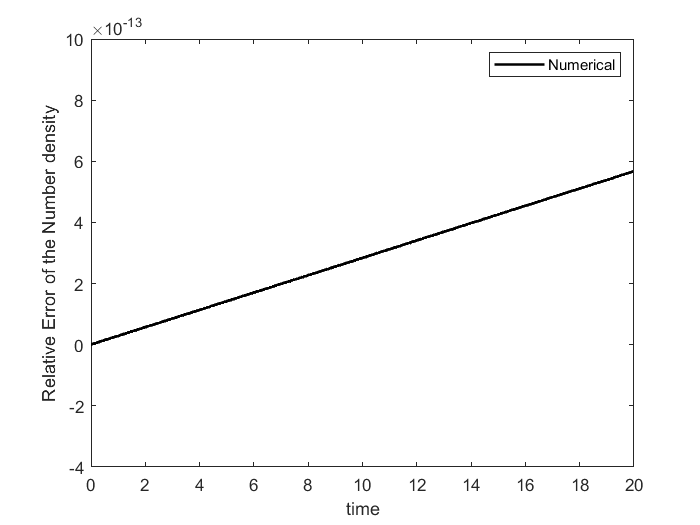}
    \caption*{(a)}
    \includegraphics[width=1.05\textwidth]{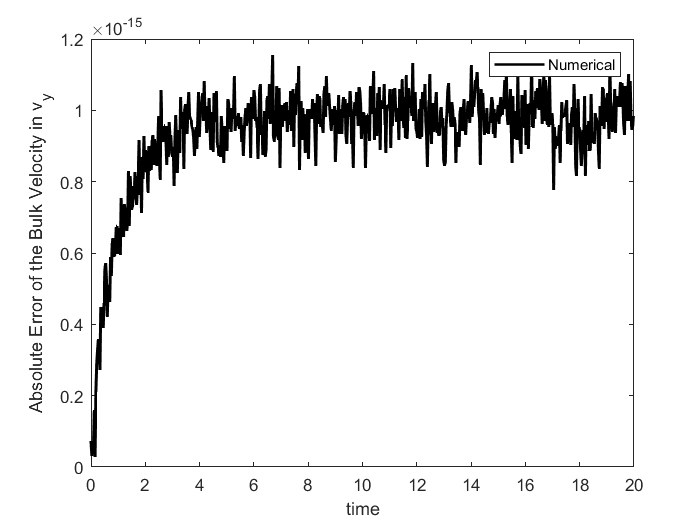}
    \caption*{(d)}
\end{minipage}
\begin{minipage}[b]{0.33\linewidth}
    \centering
    \includegraphics[width=1.05\textwidth]{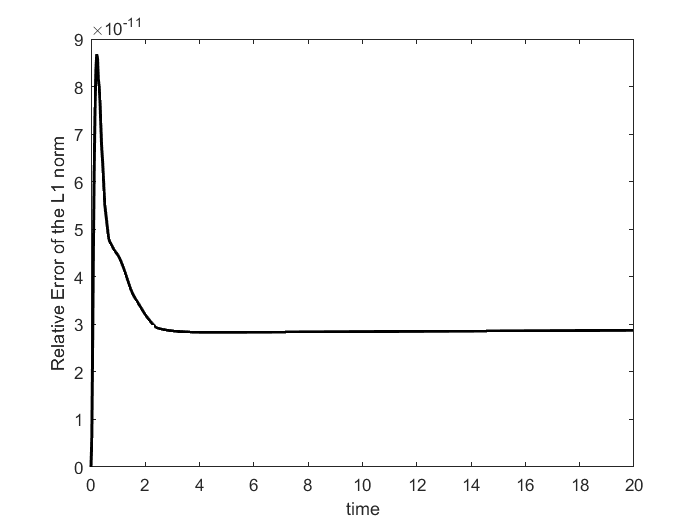}
    \caption*{(b)}
    \includegraphics[width=1.05\textwidth]{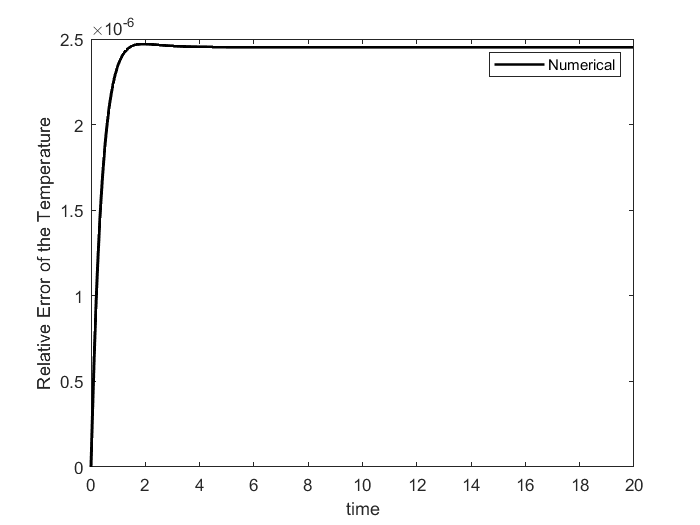}
    \caption*{(e)}
\end{minipage}
\begin{minipage}[b]{0.33\linewidth}
    \centering
    \includegraphics[width=1.05\textwidth]{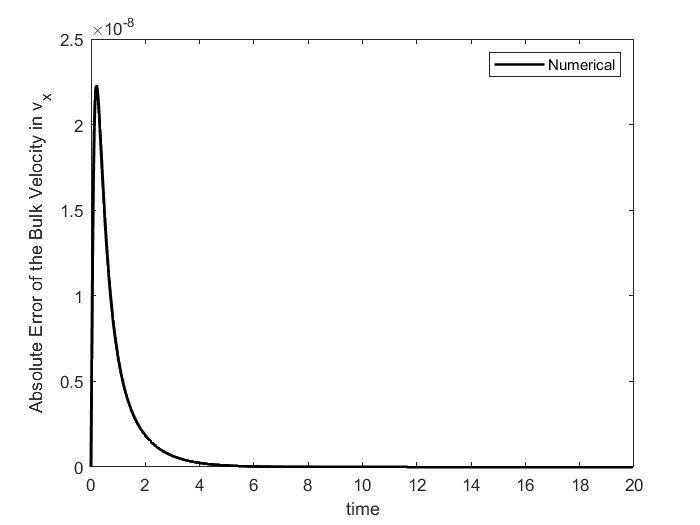}
    \caption*{(c)}
    \includegraphics[width=1.05\textwidth]{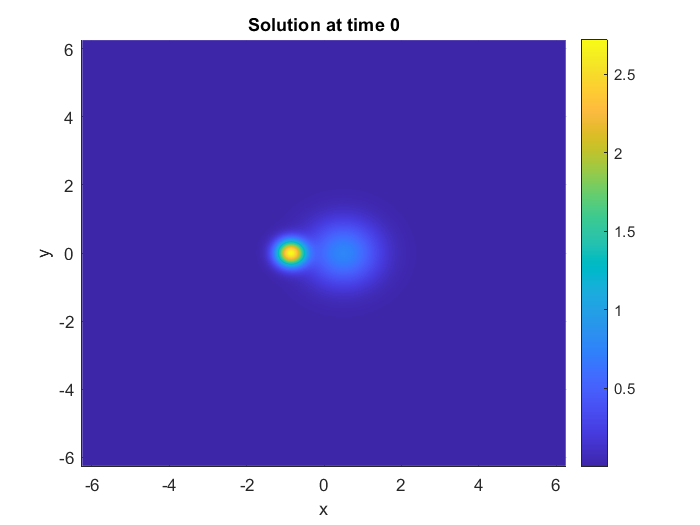}
    \caption*{(f)}
\end{minipage}
\caption{Figures (a)-(e): Relative macro-parameters for equation \eqref{2DFP} with initial distribution of two Maxwellians defined by Table \ref{2maxwells}. Mesh $N_{v_x}=N_{v_y}=200$, $CFL=6$. Figure (f): The initial distribution.}
\label{FPtwoMaxwell_images}
\end{figure}

\begin{figure}[h!]
\begin{minipage}[b]{0.33\linewidth}
    \centering
    \includegraphics[width=1.05\textwidth]{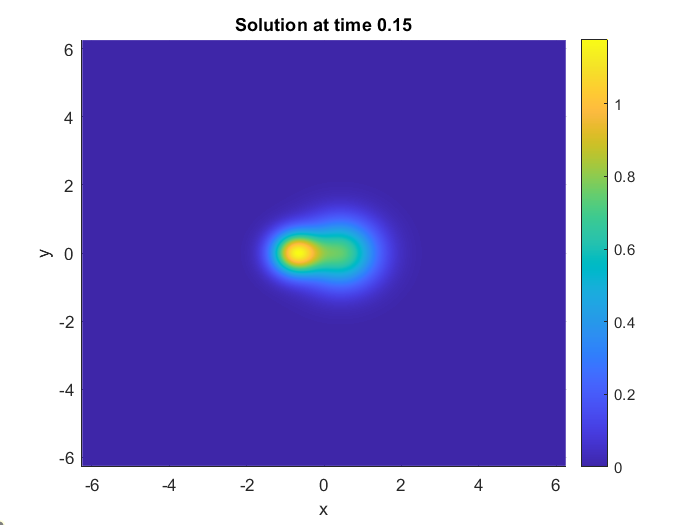}
    \includegraphics[width=1.05\textwidth]{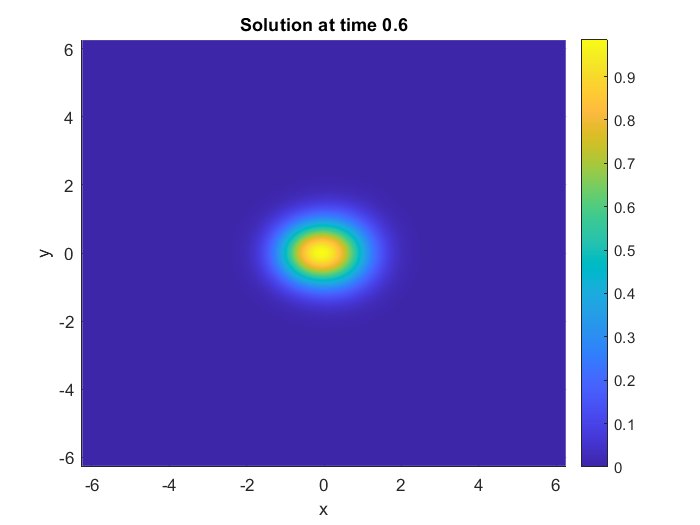}
\end{minipage}
\begin{minipage}[b]{0.33\linewidth}
    \centering
    \includegraphics[width=1.05\textwidth]{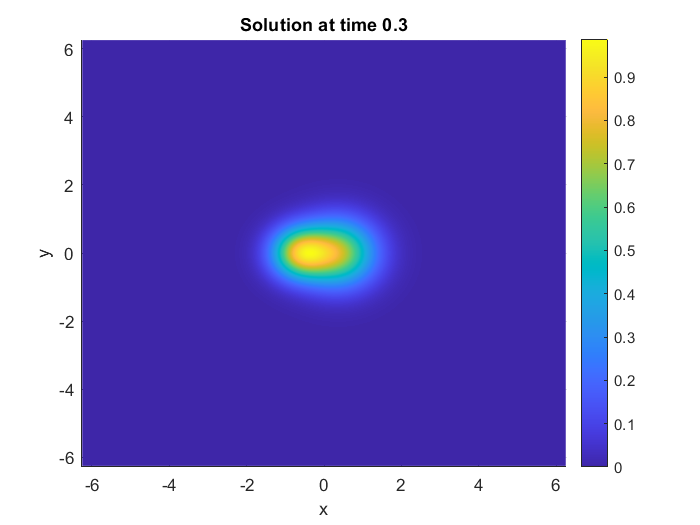}
    \includegraphics[width=1.05\textwidth]{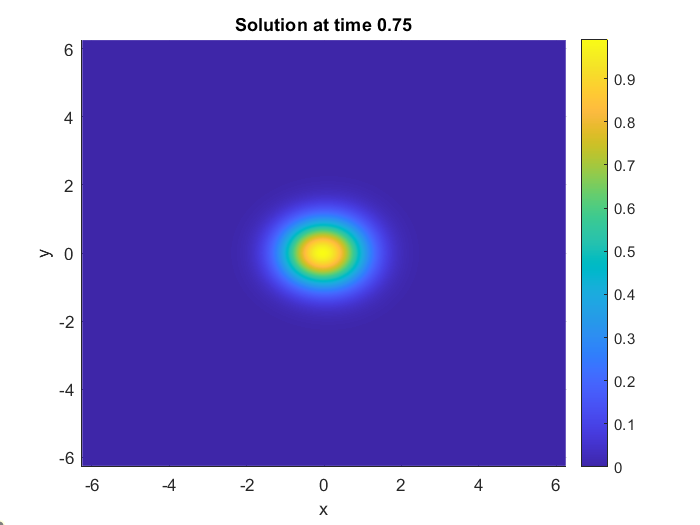}
\end{minipage}
\begin{minipage}[b]{0.33\linewidth}
    \centering
    \includegraphics[width=1.05\textwidth]{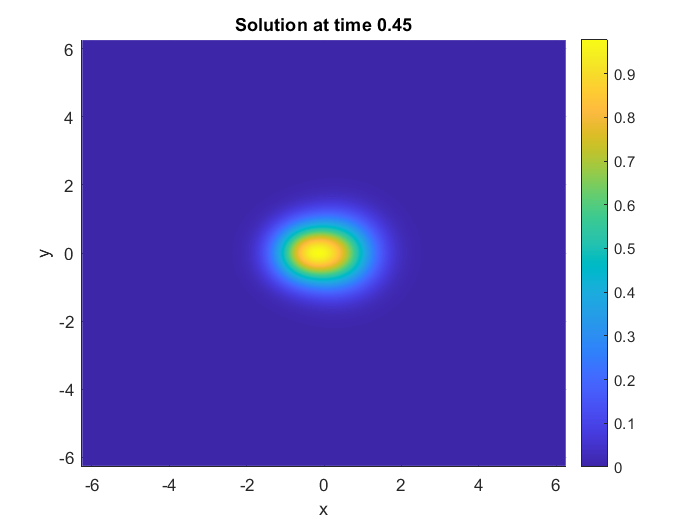}
    \includegraphics[width=1.05\textwidth]{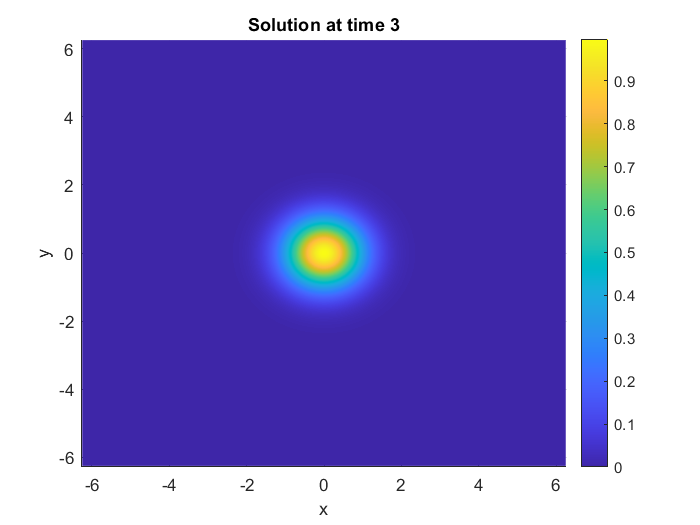}
\end{minipage}
\caption{Various snapshots of the numerical solution to equation \eqref{2DFP} with initial distribution of two Maxwellians defined by Table \ref{2maxwells}. Mesh $N_{v_x}=N_{v_y}=200$, $CFL=6$. Times: 0.15, 0.30, 0.45, 0.60, 0.75, 3.}
\label{FPtwoMaxwell_snapshots}
\end{figure}


\section{Conclusion}

In this paper, we proposed a new EL-RK-FV method for solving convection and  convection-diffusion equations. Whereas SL methods require solving for the exact characteristics, which is often highly nontrivial for nonlinear problems, our EL method computes linear space-time curves as the approximate characteristics. WENO-AO schemes allowed us to perform spatial reconstruction at arbitrary points which was essential since the traceback grid was not necessarily the (uniform) background grid. By working with the time-differential form, we could use a method-of-lines approach. Explicit RK methods were used for pure convection problems, and IMEX RK methods were used for convection-diffusion equations. Dimensional splitting was used for higher-dimensional problems. Several one- and two-dimensional test problems demonstrated the algorithm's robustness, high-order accuracy, and ability to allow extra large time steps. Ongoing and future work includes modifying the algorithm to handle shocks and rarefaction waves (the authors already have promising results that will be written in another paper), and developing a non-splitting version of the EL-RK-FV algorithm.


\section*{Acknowledgements}

Research is supported by NSF grant NSF-DMS-1818924 and NSF-DMS-2111253, Air Force Office of Scientific Research FA9550-18-1-0257 and University of Delaware.
The authors would like to thank William Taitano (Air Force Research Laboratory) and Alexander Alekseenko (California State University at Northridge) for their help in constructing the Leonard-Bernstein Fokker-Planck equation test problems. Further thanks goes to Robert Martin (Army Research Laboratory) and Alexander Alekseenko for their mentorship during Joseph Nakao's summer internship at the Air Force Research Laboratory.

\clearpage

\section*{Appendices}

\subsection*{Appendix A. Butcher Tables for Explicit RK Methods}

\begin{table}[h!]
    \centering
    \begin{minipage}[b]{0.49\linewidth}
    \centering
    \caption*{SSP RK3}
    \begin{tabular}{c|lll}
        0&0&0&0\\
        1&1&0&0\\
        1/2&1/4&1/4&0\\
        \hline
        &1/6&1/6&2/3\\
    \end{tabular}
    \end{minipage}
    \begin{minipage}[b]{0.49\linewidth}
    \centering
    \caption*{RK4}
    \begin{tabular}{c|llll}
        0&0&0&0&0\\
        1/2&1/2&0&0&0\\
        1/2&0&1/2&0&0\\
        1&0&0&1&0\\
        \hline
        &1/6&1/3&1/3&1/6\\
    \end{tabular}
    \end{minipage}
\end{table}

\subsection*{Appendix B. Fourth-order operator splitting}
Define two constants
\[\gamma_1=\frac{1}{2-2^{1/3}}\approx 1.351207191959658\quad\text{and}\quad\gamma_2=\frac{-2^{1/3}}{2-2^{1/3}}\approx -1.702414383919315.\]
Given constants $\gamma_1$ and $\gamma_2$, the fourth-order splitting method in \cite{Forest1990,Yoshida1990} has seven stages, compared to the three stages required for Strang splitting.\\
\ \\
\textbf{Step 1 ($x-$direction).} Solve equation \eqref{Strang1} over a time step $\gamma_1\Delta t/2$.\\
\textbf{Step 2 ($y-$direction).} Solve equation \eqref{Strang2} over a time step $\gamma_1\Delta t$.\\
\textbf{Step 3 ($x-$direction).} Solve equation \eqref{Strang1} over a time step $(\gamma_1+\gamma_2)\Delta t/2$.\\
\textbf{Step 4 ($y-$direction).} Solve equation \eqref{Strang2} over a time step $\gamma_2\Delta t$.\\
\textbf{Step 5 ($x-$direction).} Solve equation \eqref{Strang1} over a time step $(\gamma_1+\gamma_2)\Delta t/2$.\\
\textbf{Step 6 ($y-$direction).} Solve equation \eqref{Strang2} over a time step $\gamma_1\Delta t$.\\
\textbf{Step 7 ($x-$direction).} Solve equation \eqref{Strang1} over a time step $\gamma_1\Delta t/2$.\\
\ \\
Note that steps 2 and 6 require steps larger than $\Delta t$, and steps 3, 4, and 5 require steps backwards in time.

\subsection*{Appendix C. Butcher Tables for IMEX RK Methods}

All IMEX RK schemes included in this appendix are taken from \cite{Ascher1997}. By construction, each IMEX RK scheme has slightly different properties that are better suited from different problems. Some schemes might have better damping properties and stability regions, be stiffly accurate, etc. In this appendix we opt to pad the implicit Butcher tables with zeros.

\begin{table}[h!]
    \centering
    \begin{minipage}[b]{0.49\linewidth}
    \centering
    \caption*{IMEX(1,1,1) -- Implicit Table}
    \begin{tabular}{c|ll}
        0&0&0\\
        1&0&1\\
        \hline
        &0&1
        \end{tabular}
    \end{minipage}
    \begin{minipage}[b]{0.49\linewidth}
    \centering
    \caption*{IMEX(1,1,1) -- Explicit Table}
    \begin{tabular}{c|ll}
        0&0&0\\
        1&1&0\\
        \hline
        &1&0
        \end{tabular}
    \end{minipage}
\end{table}

\begin{table}[h!]
    \centering
    \begin{minipage}[b]{0.49\linewidth}
    \centering
    \caption*{IMEX(1,2,2) -- Implicit Table}
    \begin{tabular}{c|ll}
        0&0&0\\
        1/2&0&1/2\\
        \hline
        &0&1
        \end{tabular}
    \end{minipage}
    \begin{minipage}[b]{0.49\linewidth}
    \centering
    \caption*{IMEX(1,2,2) -- Explicit Table}
    \begin{tabular}{c|ll}
        0&0&0\\
        1/2&1/2&0\\
        \hline
        &0&1
        \end{tabular}
    \end{minipage}
\end{table}

\begin{table}[h!]
    \centering
    \begin{minipage}[b]{0.49\linewidth}
    \centering
    \caption*{IMEX(2,2,2) -- Implicit Table}
    \begin{tabular}{c|lll}
        0&0&0&0\\
        $\gamma$&0&$\gamma$&0\\
        1&0&$1-\gamma$&$\gamma$\\
        \hline
        &0&$1-\gamma$&$\gamma$
        \end{tabular}
    \end{minipage}
    \begin{minipage}[b]{0.49\linewidth}
    \centering
    \caption*{IMEX(2,2,2) -- Explicit Table}
    \begin{tabular}{c|lll}
        0&0&0&0\\
        $\gamma$&$\gamma$&0&0\\
        1&$\delta$&$1-\delta$&0\\
        \hline
        &$\delta$&$1-\delta$&0
        \end{tabular}
    \end{minipage}
    \caption*{Let $\gamma=1-\sqrt{2}/2$ and $\delta = 1 - 1/(2\gamma)$.}
\end{table}

\begin{table}[h!]
    \centering
    \begin{minipage}[b]{0.49\linewidth}
    \centering
    \caption*{IMEX(2,3,3) -- Implicit Table}
    \begin{tabular}{c|lll}
        0&0&0&0\\
        $\gamma$&0&$\gamma$&0\\
        $1-\gamma$&0&$1-2\gamma$&$\gamma$\\
        \hline
        &0&1/2&1/2
        \end{tabular}
    \end{minipage}
    \begin{minipage}[b]{0.49\linewidth}
    \centering
    \caption*{IMEX(2,3,3) -- Explicit Table}
    \begin{tabular}{c|lll}
        0&0&0&0\\
        $\gamma$&$\gamma$&0&0\\
        $1-\gamma$&$\gamma-1$&$2(1-\gamma)$&0\\
        \hline
        &0&1/2&1/2
        \end{tabular}
    \end{minipage}
    \caption*{Let $\gamma=(3+\sqrt{3})/6$.}
\end{table}

\begin{table}[h!]
    \centering
    \begin{minipage}[b]{0.49\linewidth}
    \centering
    \caption*{IMEX(2,3,2) -- Implicit Table}
    \begin{tabular}{c|lll}
        0&0&0&0\\
        $\gamma$&0&$\gamma$&0\\
        1&0&$1-\gamma$&$\gamma$\\
        \hline
        &0&$1-\gamma$&$\gamma$
        \end{tabular}
    \end{minipage}
    \begin{minipage}[b]{0.49\linewidth}
    \centering
    \caption*{IMEX(2,3,2) -- Explicit Table}
    \begin{tabular}{c|lll}
        0&0&0&0\\
        $\gamma$&$\gamma$&0&0\\
        1&$\delta$&$1-\delta$&0\\
        \hline
        &0&$1-\gamma$&$\gamma$
        \end{tabular}
    \end{minipage}
    \caption*{Let $\gamma=(2-\sqrt{2})/2$ and $\delta = -2\sqrt{2}/3$.}
\end{table}

\begin{table}[h!]
    \centering
    \begin{minipage}[b]{0.49\linewidth}
    \centering
    \caption*{IMEX(3,4,3) -- Implicit Table}
    \begin{tabular}{c|llll}
        0&0&0&0&0\\
        $\gamma$&0&$\gamma$&0&0\\
        0.717933&0&0.282067&$\gamma$&0\\
        1&0&1.208497&-0.644363&$\gamma$\\
        \hline
        &0&1.208497&-0.644363&$\gamma$
    \end{tabular}
    \end{minipage}
    \begin{minipage}[b]{0.49\linewidth}
    \centering
    \caption*{IMEX(3,4,3) -- Explicit Table}
    \begin{tabular}{c|llll}
        0&0&0&0&0\\
        $\gamma$&$\gamma$&0&0&0\\
        0.717933&0.321279&0.396654&0&0\\
        1&-0.105858&0.552929&0.552929&0\\
        \hline
        &0&1.208497&-0.644363&$\gamma$
    \end{tabular}
    \end{minipage}
    \caption*{Let $\gamma=0.435867$.}
\end{table}

\begin{table}[h!]
    \centering
    \begin{minipage}[b]{0.49\linewidth}
    \centering
    \caption*{IMEX(4,4,3) -- Implicit Table}
    \begin{tabular}{c|lllll}
        0&0&0&0&0&0\\
        1/2&0&1/2&0&0&0\\
        2/3&0&1/6&1/2&0&0\\
        1/2&0&-1/2&1/2&1/2&0\\
        1&0&3/2&-3/2&1/2&1/2\\
        \hline
        &0&3/2&-3/2&1/2&1/2
        \end{tabular}
    \end{minipage}
    \begin{minipage}[b]{0.49\linewidth}
    \centering
    \caption*{IMEX(4,4,3) -- Explicit Table}
    \begin{tabular}{c|lllll}
        0&0&0&0&0&0\\
        1/2&1/2&0&0&0&0\\
        2/3&11/18&1/18&0&0&0\\
        1/2&5/6&-5/6&1/2&0&0\\
        1&1/4&7/4&3/4&-7/4&0\\
        \hline
        &1/4&7/4&3/4&-7/4&0
        \end{tabular}
    \end{minipage}
\end{table}

\clearpage
\subsection*{Appendix D. An illustrative example with IMEX(2,2,2)}

In this section, we couple the EL-RK-FV algorithm with IMEX(2,2,2), that is, two-stage implicit, two-stage explicit, and of combined order two. This scheme is L-stable and uses a second order DIRK method. The Butcher tables are given in Appendix C. Figure \ref{IMEX222} shows the lone sub-space-time region ${}_{1}\Omega_j$.

\begin{figure}[h!]
	\centering
	\includegraphics[width=0.5\textwidth]{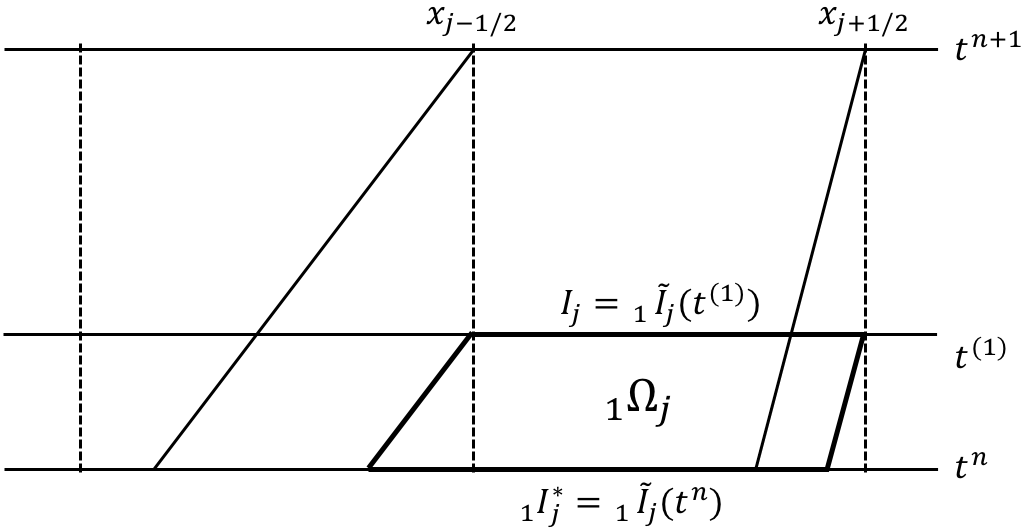}
	\caption{The space-time region ${}_{1}\Omega_j$ for IMEX(2,2,2).}
	\label{IMEX222}
\end{figure}

\noindent\textbf{Step 0a.} Compute the approximate characteristic speeds using equation \eqref{RHjump}. After defining the space-time region $\Omega_j$, compute the possibly nonuniform traceback cell averages $\tilde{u}_j(t^n)$ using Algorithm 1.\\
\textbf{Step 0b.} Use the possibly nonuniform traceback cell averages $\tilde{u}_j(t^n)$ in Algorithm 2 to compute $\hat{K}_1 = \mathcal{F}(U^n;t^n))$.\\
\textbf{Step 1a.} Using the same approximate characteristic speeds from step 1a, define the sub-space-time region ${}_{1}\Omega_j$ as seen in Figure \ref{IMEX222}. Compute the possibly nonuniform traceback cell averages ${}_{1}\tilde{u}_j^n$ using Algorithm 1.\\
\textbf{Step 1b.} Use the possibly nonuniform traceback cell averages ${}_{1}\tilde{u}_j^n$ in Algorithm 2 to compute ${}_{1}\hat{K}_1 = \mathcal{F}({}_{1}U^n;t^n))$.\\
\textbf{Step 1c.} Recalling equation \eqref{5stencil}, solve equation \eqref{IMEXsubstage} by solving the linear system
\begin{equation}
    \left(\mathbf{I}-\frac{\gamma\epsilon\Delta t}{\Delta x^2}\mathbf{D}_4\right){}_{1}\vv{U}^{(1)} = {}_{1}\vv{U}^n + \gamma\Delta t{}_{1}\vv{\hat{K}}_1 + \gamma\Delta t\vv{g}(x,t^{(1)}),
\end{equation}
where $\vv{g}_j(x,t^{(1)}) = \int_{I_j}{g(x,t^{(1)})dx}$ can be computed with a Gaussian quadrature.\\
\textbf{Step 1d.} Compute the uniform cell averages $\overline{u}_j^{(1)} = {}_{1}U_j^{(1)}/\Delta x$.\\
\textbf{Step 1e.} Compute the uniform cell averages $\overline{u}_{xx,j}^{(1)}$ using equation \eqref{5stencil},
\begin{equation}
    \vv{\overline{u}}_{xx}^{(1)} = \frac{1}{\Delta x^2}\mathbf{D}_4\vv{\overline{u}}^{(1)}.
\end{equation}
\textbf{Step 1f.} Compute the possibly nonuniform traceback cell averages $\tilde{u}_j^{(1)}$ and $\tilde{u}_{xx,j}^{(1)}$ (we are now in the space-time region $\Omega_j$) using Algorithm 1.\\
\textbf{Step 1g.} Compute $K_1 = \mathcal{G}(U^{(1)};t^{(1)})$,
\begin{equation}
    K_1 = \epsilon\Delta\tilde{x}_j^{(1)}\tilde{u}_{xx,j}^{(1)} + \int_{\tilde{I}_j(t^{(1)})}{g(x,t^{(1)})dx},
\end{equation}
where the definite integral involving $g(x,t)$ can be evaluated using a Gaussian quadrature.\\
\textbf{Step 1h.} Use the possibly nonuniform traceback cell averages $\tilde{u}_j^{(1)}$ in Algorithm 2 to compute $\hat{K}_2 = \mathcal{F}(U^{(1)};t^{(1)})$.\\
\textbf{Step 2.} Recalling equation \eqref{5stencil}, solve equation \eqref{IMEX} by solving the linear system
\begin{equation}
    \left(\mathbf{I}-\frac{\gamma\epsilon\Delta t}{\Delta x^2}\mathbf{D}_4\right)\vv{U}^{n+1} = \vv{U}^n + (1-\gamma)\Delta t\vv{\hat{K}}_1 + \Delta t(\delta\vv{\hat{K}}_1 + (1-\delta)\vv{\hat{K}}_2) + \gamma\Delta t\vv{g}(x,t^{n+1}),
\end{equation}
where $\vv{g}_j(x,t^{n+1}) = \int_{I_j}{g(x,t^{n+1})dx}$ can be computed with a Gaussian quadrature.



\end{document}